\documentclass[12pt,a4paper]{amsart}
\usepackage[utf8]{inputenc}
\usepackage{amsmath,amssymb,amsthm,xcolor,qtree,bbm,centernot,mathrsfs}
\usepackage[margin=2.5cm]{geometry}
\usepackage{hyperref}

\usepackage{enumitem}
\usepackage{comment}
\usepackage[normalem]{ulem}

\title[Separation of $\ell_{1}$ spreading models and the Lebesgue property]{On the complete separation of unique $\ell_{1}$ spreading models and the Lebesgue property of Banach spaces} 
\author{Harrison Gaebler} 
\author{Pavlos Motakis}
\author{B\"unyamin Sar\i}
\address{Department of Mathematics, University of North Texas, 1155 Union Circle 311430
Denton, Texas 76203-5017}
\address{Department of Mathematics and Statistics, York University, 4700 Keele Street, Toronto, Ontario, M3J 1P3, Canada}
\email{harrison.gaebler@unt.edu, pmotakis@yorku.ca, bunyamin.sari@unt.edu}

\thanks{The second author was supported by NSERC Grant RGPIN-2021-03639.}


\newcommand{\N}{\mathbb{N}}
\newcommand{\bes}{\begin{equation*}}
\newcommand{\ees}{\end{equation*}}
\newcommand{\ep}{\varepsilon}
\newcommand{\supp}{\text{supp}}

\theoremstyle{definition}
\newtheorem{question}{Question}[section]
\newtheorem{definition}[question]{Definition}
\theoremstyle{plain}
\newtheorem{theorem}[question]{Theorem}
\newtheorem{lemma}[question]{Lemma}
\newtheorem{proposition}[question]{Proposition}
\newtheorem{corollary}[question]{Corollary}

\newtheorem{remark}[question]{Remark}

\begin{document}

\maketitle

\begin{abstract}
We construct a reflexive Banach space $X_\mathcal{D}$ with an unconditional basis such that all spreading models admitted by normalized block sequences in $X_\mathcal{D}$ are uniformly equivalent to the unit vector basis of $\ell_1$, yet every infinite-dimensional closed subspace of $X_\mathcal{D}$ fails the Lebesgue property. This is a new result in a program initiated by Odell in 2002 concerning the strong separation of asymptotic properties in Banach spaces.
\end{abstract}

\section{Introduction}{\label{sec1}}

In Banach space theory, an asymptotic notion refers to a concept describing the behavior of the norm of linear combinations of vectors sampled in a specified asymptotic fashion from a certain type of structure. For example, the seminal notion of a spreading model of a bounded sequence $(x_n)_{n=1}^\infty$ is defined as the unit vector basis $(e_i)_{i=1}^\infty$ of $c_{00}$ with the seminorm given by
\[\Big|\sum_{i=1}^na_i e_i\Big| = \lim_{k_1\to\infty}\lim_{k_{2}\to\infty}\cdots\lim_{k_n\to\infty}\Big\|\sum_{i=1}^na_ix_{k_i}\Big\|,\]
provided that this iterated limit exists for all choices of coefficients $a_1,\ldots,a_n$. Here, the structure in question is a bounded sequence of vectors in a Banach space, and vectors are sampled over a sparse set of sufficiently large indices. Using Ramsey's Theorem, Brunel and Sucheston proved in \cite{brunel:sucheston:1974} that every bounded sequence in a Banach space has a subsequence generating some spreading model. Similar notions about different types of structures have been introduced and studied over the years. One such example is the notion of asymptotic models, introduced by Odell and Halbeisen in \cite{halbeisen:odell:2004}, concerning the norm of linear combinations of vectors sampled from an array of sequences. A different type of such a notion is asymptotic spaces, introduced by Maurey, Milman, and Tomczak-Jaegerman in \cite{maurey:milman:tomczak-jaegermann:1995}, concerning the norm of linear combinations of vectors picked in a finite-round two-player game between a vector chooser and a finite-codimensional subspace chooser. The different types of asymptotic notions are relevant to the study of many subjects in Banach space theory, such as the distortion problem in \cite{milman:tomczak-jaegermann:1993} using asymptotic spaces, the invariant subspace problem in reflexive Banach spaces in \cite{argyros:motakis:2014} using spreading models, the uniform approximation properties of bounded linear operators in \cite{argyros:georgiou:lagos:motakis:2020} using joint spreading models, the coarse geometry of Banach spaces in \cite{baudier:lancien:motakis:schlumprecht:2021} using asymptotic models, unique maximal ideals in the operators of stopping time-type Banach spaces in \cite{kania:lechner:2022} using spreading models, and quotients algebras of operator spaces in \cite{motakis:pelczar-barwacz:2024} using asymptotic versions. The wide applicability of these tools has led to an independent study of their uniqueness properties and a desire to better understand their interconnections. Among the most outspoken promoters of this program was Odell, who in this context posed several questions of the following type (see \cite{odell:2002}, \cite{odell:2007}, and \cite{junge:kutzarova:odell:2006}): Let $P$, $Q$ be properties of Banach spaces that are stable under passing to closed subspaces and assume that $P$ implies $Q$. If a Banach space $X$ satisfies Q must it have an infinite-dimensional closed subspace satisfying $P$? For such $P$ and $Q$, if there exists a Banach space $X$ satisfying $Q$ and all infinite-dimensional subspaces of which fail $P$ then we say that $P$ is completely separated from $Q$. For a detailed discussion on positive and negative results in this context, we refer the reader to \cite{argyros:motakis:2020}. We prove such a complete separation theorem for $P$ the Lebesgue property and $Q$    the property of admitting a unique $\ell_1$ spreading model.

Let $X$ be a Banach space. A function $f:[0,1]\to X$ is Riemann-integrable if it is bounded and continuous up to a set of Lebesgue measure zero (see, e.g., \cite[Theorem 18]{gordon:1991} or \cite[Theorem. 2.1.3.]{gaebler:2021}). However, the converse statement is false in general. Indeed, it is easy to check that the function which maps the rationals in $[0,1]$ to the unit vector basis of $c_{0}$ is Riemann-integrable. A Banach space $X$ is therefore said to have the Lebesgue property if every Riemann-integrable function $f:[0,1]\to X$ is continuous up to a set of Lebesgue measure zero. The first and third named authors recently in \cite{gaebler:sari:2023} characterized the Lebesgue property in terms of a new asymptotic notion. In particular, a Banach space $X$ has the Lebesgue property if and only if every normalized basic sequence in $X$ is Haar-$\ell_{1}^{+}$ (see \cite{gaebler:sari:2023} or Section \ref{sec2} of this paper for the precise definition of Haar-$\ell_{1}^{+}$). It is also proved in \cite{gaebler:sari:2023} that for the properties:
\begin{enumerate}[label=(P\arabic*)]
\item\label{intro unique ell1 am} every asymptotic model of $X$ is equivalent to the unit vector basis of $\ell_{1}$ 
\item\label{intro lebesgue property} $X$ satisfies the Lebesgue property
\item\label{intro unique ell1 sm} every spreading model of $X$ is equivalent to the unit vector basis of $\ell_{1}$
\end{enumerate}
the implications \ref{intro unique ell1 am}$\implies$\ref{intro lebesgue property}$\implies$\ref{intro unique ell1 sm} are true. On the other hand, \ref{intro lebesgue property}$\centernot{\implies}$\ref{intro unique ell1 am} because the Tsirelson sum of Tsirelson spaces, $(T\oplus T\oplus\ldots)_{T}$, has the Lebesgue property but no unique asymptotic model (see, e.g., \cite{{gaebler:sari:2023}}) and \ref{intro unique ell1 sm}$\centernot{\implies}$\ref{intro lebesgue property} because there exist Schur spaces that do not have the Lebesgue property (e.g. \cite{haydon:1984},\cite{naralenkov:2008}) and Schur spaces satisfy \ref{intro unique ell1 sm}. The complete separation of \ref{intro unique ell1 am} from \ref{intro lebesgue property} was proved in \cite{gaebler:sari:2023} by showing that the Banach space $X_{iw}$ constructed in \cite{argyros:motakis:2020} has the Lebesgue property, and yet, every subspace of $X_{iw}$ contains a $c_{0}$-asymptotic model. The complete separation of \ref{intro lebesgue property} from \ref{intro unique ell1 sm} was not attempted in \cite{gaebler:sari:2023}, and it is the purpose of the present article. By carrying out a substantial modification of the construction in \cite{argyros:motakis:2020}, we are able to achieve this new result. We construct a Banach space $X_\mathcal{D}$ with an unconditional basis such that every spreading model of $X_\mathcal{D}$ is uniformly equivalent to the unit vector basis of $\ell_{1}$ and such that every subspace of $X_\mathcal{D}$ fails the Lebesgue property.

Let us explain the main properties of $X_\mathcal{D}$. Denote by $\mathcal{D}$ the binary tree $\cup_{n=0}^\infty\{0,1\}^n$.  Following standard convention, $\{0,1\}^0$ is a singleton containing the empty sequence. For $n\in\mathbb{N}\cup \{0\}$ and $\lambda\in\{0,1\}^n$ we write $|\lambda| = n$ and say $\lambda$ has height $n$. We consider $\mathcal{D}$ with the usual partial order with minimum the empty sequence and, otherwise, $\lambda \leq \mu$ means $|\lambda|\leq |\mu|$ and $\lambda(i) = \mu(i)$, for $1\leq i\leq |\lambda|$.  The characterization of the Lebesgue property in terms of Haar-$\ell_1^+$ sequences from \cite{gaebler:sari:2023} easily yields the following (see Remark \ref{lebesgue on binary trees}): a Banach space $X$ has the Lebesgue property if and only if, {for every collection $(x_\lambda)_{\lambda\in\mathcal{D}}$ of norm-one vectors in $X$, there exists a constant $\theta>0$ such that, for every $n\in\mathbb{N}$, there exists a subset $V\subset\{0,1\}^{n}$ such that
\bes\theta\leq \frac{1}{2^{n}}\left\Vert\sum_{\lambda\in V} x_{\mu_{\lambda}}\right\Vert \ees
for some choice of nodes $(\mu_{\lambda})_{\lambda\in V}$ with $\mu_{\lambda}\geq\lambda$.}
In the space $X_\mathcal{D}$, every infinite-dimensional subspace contains a lexicographically block collection $(x_\lambda)_{{\lambda\in\mathcal{D}}}$ of norm-one vectors such that for any $n\in\N$ and collection $(\mu_\lambda)_{\lambda\in\{0,1\}^n}$ in $\mathcal{D}$ such that $\mu_\lambda\geq\lambda$, $\lambda\in\{0,1\}^n$,
\begin{equation}
\label{intro ellinfty estimate}
\max_{\lambda\in\{0,1\}^n}|a_\lambda|\leq \Big\|\sum_{{\lambda\in\{0,1\}^{n}}} a_\lambda x_{\mu_\lambda}\Big\|  \leq 3  \max_{{\lambda\in\{0,1\}^n}}|a_\lambda|, 
\end{equation}
for any scalars $(a_\lambda)_{{\lambda\in\{0,1\}^n}}$. Therefore, $X_\mathcal{D}$ fails the Lebesgue property. At the same time, every normalized weakly null sequence in $X_\mathcal{D}$ has a subsequence generating a spreading model that is $102$-equivalent to the unit vector basis of $\ell_1$. In particular, it {follows from} James's characterization of reflexivity for spaces with an unconditional basis that $X_\mathcal{D}$ is reflexive.

Our construction relies on the method of saturation under constraints with increasing weights, also employed in \cite{argyros:motakis:2020}. The major difference lies in the involvement of an additional metric constraint. A distantly related version of this constraint was also present in \cite{motakis:2024}, but its application and role were significantly different as it produced an operator-algebraic outcome (see \cite[Section 2]{motakis:2024}). The earliest form of saturation under constraints is due to Odell and Schlumprecht (see \cite{odell:schlumprecht:1995} and \cite{odell:schlumprecht:2000}). It has since been extensively developed, e.g., in the papers \cite{argyros:motakis:2014}, \cite{beanland:freeman:motakis:2015}, \cite{argyros:motakis:2020}, and many others. Let us elaborate on the main ingredients of our construction. The norm of the space $X_\mathcal{D}$ is defined via a norming set, i.e., there is a subset $W$ of $c_{00}$ containing the unit vector basis $(e_i^*)_{i=1}^\infty$ such that, for $x\in c_{00}$, $\|x\| = \sup\{|\langle f,x\rangle|:f\in W\}$, where $\langle\cdot,\cdot\rangle$ denotes the standard duality pairing of $c_{00}\times c_{00}$. The space $X_\mathcal{D}$ is the completion of $c_{00}$ under this norm. We fix a pair of lacunary sequences $m_j$, $n_j$, $j\in\N$, of natural numbers. The first sequence is used to define a collection of weights $\{\prod_{i=1}^\ell m_{j_i}:j_1,\ldots,j_\ell\in\N\}$, which is closed under multiplication. To each weight $\prod_{i=1}^\ell m_{j_i}$, we associate a Schreier family $\mathcal{S}_{n_{j_1}+n_{j_2}+\cdots+n_{j_\ell}}$. This is a compact collection of finite subsets of $\N$ recalled in Section \ref{sec2}. Every $f$ in $W$ satisfies $\|f\|_\infty \leq 1$ and it is either of the form $\pm e_i^*$ or 
\begin{equation}
\label{intro functional form}
f = \frac{1}{w(f)}\sum_{q=1}^d f_q,
\end{equation}
where $w(f) = \prod_{i=1}^\ell m_{j_i}$ and $f_1,\ldots,f_d$ are successively supported vectors such that the set $\{\min\mathrm{supp}(f_q):1\leq q\leq d\}$ is in the collection $\mathcal{S}_{n_{j_1}+n_{j_2}+\cdots+n_{j_d}}$. A Banach space defined by a norming set on which only the aforementioned restrictions are imposed is asymptotic-$\ell_1$ (see, e.g., \cite{argyros:deliyanni:1997}), an asymptotic property implying having a unique $\ell_1$ asymptotic model and thus also the Lebesgue property. To avoid this without compromising the uniform uniqueness of $\ell_1$ spreading models, just as in \cite{argyros:motakis:2020}, we additionally demand that for $f$ in $W$ as in \eqref{intro functional form}, the sequence $w(f_1),\ldots,w(f_d)$ is increasing sufficiently rapidly in relation to the supports of $f_1$,\ldots,$f_d$. Imposing these additional restrictions on $W$ yields a Banach space with a uniformly unique $\ell_1$ spreading model and no asymptotic-$\ell_1$ subspace; however, as it was proved in \cite{gaebler:sari:2023}, it still satisfies the Lebesgue property. The intuitive explanation for this is the following: despite the imposed restrictions on $W$, there are sufficiently many functionals yielding the required $\ell_1^+$ estimates on normalized collections $(x_\lambda)_{\lambda\in\mathcal{D}}$ in the defined Banach space.

To address this, we impose an additional metric constraint on $W$. Recall that $\mathcal{D}$ is a totally bounded metric space with the metric $d(\lambda,\mu) = 2^{-n}$, where, if $\lambda\neq\mu$, $n\in\N\cup\{0\}$ is the maximum height for which there exists $\nu\in\{0,1\}^n$ with $\nu\leq \lambda$ and $\nu\leq \mu$. We fix a bijection $\phi:\{m_j:j\in\N\}\to \mathcal{D}$ and to every $f\in W$ of weight $\prod_{i=1}^\ell m_{j_i}$ we associate the node $\lambda = \phi(m(f))$, where $m(f)$ is the maximum of $m_{j_1},\ldots,m_{j_\ell}$. We then impose on $W$ the restriction that for every $f$ as in \eqref{intro functional form}, the nodes $\phi(m(f_1)),\ldots,\phi(m(f_d))$ approximate some $\lambda$ in the completion of $\mathcal{D}$ and the quantification of the approximation depends on the tuple $(n_{j_1},\ldots, n_{j_\ell})$ and the supports of the $f_q$, $1\leq q\leq d$. Modulo some subtle but significant details, the construction of the space $X_\mathcal{D}$ is complete.

Using a standard technique, in every block-subspace $Y$ of $X_\mathcal{D}$, we build a collection of vectors $(x_{m_j})_{j\in\mathbb{N}}$ such that, for every $j\in\mathbb{N}$, the norm of $x_{m_j}$ can only be effectively estimated by functionals of weight $w(f) = m_j$, i.e., $x_{m_j}$ is {\em an exact vector}. We then consider the $\mathcal{D}$-indexed collection $(x_{\phi^{-1}(\lambda)})_{\lambda\in\mathcal{D}}$. This witnesses the failure of the Lebesgue property in $Y$. Indeed, for $n\in\N$, denote $\{0,1\}^n = \{\lambda_1,\ldots,\lambda_{2^n}\}$ and let $m_{j_1}$,\ldots,$m_{j_{2^n}}$ such that $\phi(m_{j_1}) \geq \lambda_1$,\ldots,$\phi(m_{j_{2^n}})\geq \lambda_{2^n}$. Then the nodes $\phi(m_{j_1})$,\ldots$\phi(m_{j_{2^n}})$ form a $2^{-n}$ separated set in $\mathcal{D}$. For every $f\in W$ as in \eqref{intro functional form}, $\phi(m(f_1))$,\ldots,$\phi(m(f_d))$ approximate a $\lambda$ in the completion of $\mathcal{D}$, and thus, roughly speaking, the components of $f$ only contribute to the norm of very few of the $x_{m_{j_i}}$, $1\leq i\leq 2^n$, yielding the $\ell_\infty$-estimate \eqref{intro ellinfty estimate}.

Our paper is organized as follows. In Section \ref{sec2}, we recall the required terminology surrounding Schreier families and special convex combinations. We also recall the definition of a Haar-$\ell_1^+$ sequence and remark on the formulation of the Lebesgue property in terms of normalized collections $(x_\lambda)_{\lambda\in\mathcal{D}}$.  In Section \ref{sec3}, we define the space $X_\mathcal{D}$ and prove the uniform uniqueness of $\ell_1$ spreading models. In Section \ref{sec4}, we define an auxiliary Banach space $X_\mathcal{D}^{\mathrm{aux},\delta}$ and prove $\ell_\infty$-type estimates for linear combinations of ``simple'' vectors in this space. In Section \ref{sec5}, we define and study rapidly increasing sequences (RISs) and the basic inequality. The former comprise a special class of sequences in $X_\mathcal{D}$, and the latter provides bounds for their linear combinations in terms of linear combinations of basis vectors in the auxiliary space. In Section \ref{sec6}, we combine all ingredients prepared in Sections \ref{sec4} and \ref{sec5} to construct in every infinite-dimensional closed subspace of $X_\mathcal{D}$ a normalized collection $(x_\lambda)_{\lambda\in\mathcal{D}}$ witnessing the failure of the Lebesgue property.

\section{Preliminary information}{\label{sec2}}
Let $A,B\subset\N$. If $\max(A)<\min(B)$, then we write $A<B$. We also write $n\leq A$ if $n\in\N$ is at most $\min(A)$ and, by convention, $\emptyset<A$ and $A<\emptyset$ for every $A\subset\N$. We denote by $c_{00}$ the vector space that consists of sequences of real numbers that have at most finitely-many non-zero terms and we let $(e_{i})$ be the unit vector basis of $c_{00}$. Note that the members of the norming set $W$ are themselves members of $c_{00}$ and, in particular, that $\{\pm e_{i}\}_{i=1}^{\infty}\subset W$. For this reason, we write $e^{*}_{i}\in c_{00}$ when referring to $e_{i}$ as a norming functional. The support of $x=(c_{i})\in c_{00}$ is the set $\supp(x)=\{i\in\N \mid c_{i}\neq 0\}$ and we say that vectors $x_{1},\ldots,x_{n}\in c_{00}$ are successive if $\supp(x_{i-1})<\supp(x_{i})$ for $2\leq i\leq n$. In this case, we use the shorthand notation $x_{1}<\ldots<x_{n}$ and we also write $n\leq x$ if $n\leq\min\supp(x)$. A finite or infinite sequence of successive vectors in $c_{00}$ is said to be a block sequence of $(e_{i})$.

The Schreier sets are an increasing sequence of sets of finite subsets of $\N$. These sets are defined inductively as follows:
\bes S_{0}=\{\{i\} \mid i\in\N\} \text{ and } S_{1}=\{F\subset\N \mid \text{card}(F)\leq F\} \ees
and we write that
\bes S_{n+1}=\{ F\subset\N \mid F=\cup_{i=1}^{d}F_{i} \text{ where } d\leq F_{1}<\ldots<F_{d}\in S_{n}\} \ees
once $S_{n}$ has been defined. We also define the set $\mathcal{A}_{3}=\{ F\subset\N \mid \text{card}(F)\leq 3\}$ and we will be interested in sets of the form
\bes S_{n}\ast\mathcal{A}_{3} = \{ F\subset\N \mid F=\cup_{i=1}^{d} F_{i} \text{ where } F_{1}<\ldots<F_{d}\in\mathcal{A}_{3} \text{ and } \{\min(F_{i})\}_{i=1}^{d}\in S_{n}\} \} \ees
and it is easy to see that $F\in S_{n}\ast\mathcal{A}_{3}$ is the union of at most three sets in $S_{n}$. A block sequence $x_{1}<\ldots<x_{k}\in c_{00}$ is said to be $S_{n}$-admissible (resp. $S_{n}\ast\mathcal{A}_{3}$-admissible) if $\{\min\supp(x_{i})\}_{i=1}^{k}\in S_{n}$ (resp. $\{\min\supp(x_{i})\}_{i=1}^{k}\in S_{n}\ast\mathcal{A}_{3}$).

Let $n\in\N$ and let $\ep>0$. Then, the vector $x=\sum_{i\in F}c_{i}e_{i}\in c_{00}$ is said to be an $(n,\ep)$-basic special convex combination ($(n,\ep)$-basic scc) if
\begin{enumerate}
\item $F\in S_{n}$, $c_{i}\geq 0$ for each $i\in F$, and $\sum_{i\in F}c_{i}=1$
\item For every $G\subset F$ with $G\in S_{n-1}$, we have that $\sum_{i\in G}c_{i}<\ep$
\end{enumerate}
and we refer the reader to \cite{argyros:deliyanni:1997}, \cite{argyros:godefroy:rosenthal:2003}, or \cite{argyros:tolias:2004} for more information. In particular, note that $\sum_{i\in G}c_{i}<3\ep$ if $G\subset F$ with $G\in S_{k}\ast\mathcal{A}_{3}$ for $k<n$, and there is also the following useful result which is Proposition 2.3 of \cite{argyros:tolias:2004}.

\begin{proposition}{\label{basic-scc-maximal}}
Let $n\in\N$. Then, for every infinite subset $M\subset\N$ and for every $F\subset M$ such that $F\in S_{n}$ is maximal with respect to inclusion, there exist real numbers $(c_{i})_{i\in F}$ so that the vector $x=\sum_{i\in F}c_{i}e_{i}\in c_{00}$ is an $(n,3/\min(F))$-basic scc.
\end{proposition}

Similarly, if $x_{1}<\ldots<x_{d}\in c_{00}$ and if $s_{i}=\min\supp(x_{i})$ for each $1\leq i\leq d$, then the vector $x=\sum_{i=1}^{d}c_{i}x_{i}$ is said to be an $(n,\ep)$-special convex combination ($(n,\ep)$-scc) if $\sum_{i=1}^{d}c_{i}e_{s_{i}}$ is an $(n,\ep)$-basic scc. It follows in this case that $\sum_{i=1}^{d}c_{i}e_{t_{i}}$ is an $(n,2\ep)$-basic scc as well where $t_{i}=\max\supp(x_{i})$ for each $1\leq i\leq d$ because
\bes \sum_{i\in G}c_{i}=\sum_{i\in G_{1}}c_{i}+\sum_{i\in G_{2}}c_{i}<\ep+\ep=2\ep \ees
if $G\subset\{t_{i}\}_{i=1}^{d}$ is $S_{n-1}$-admissible and is therefore a translation to the right of at most two sets $G_{1},G_{2}\subset\{s_{i}\}_{i=1}^{d}$ with $G_{1},G_{2}\in S_{n-1}$.

Finally, two specific asymptotic structures appear in this paper. The first is that of a spreading model from \cite{brunel:sucheston:1974}, recalled in page \pageref{sec1}. It is a consequence of Ramsey's theorem that every normalized basic sequence $(x_{i})$ in a Banach space $X$ has a subsequence that generates a spreading model of $X$. We now define the notion of a Haar-$\ell_{1}^{+}$ sequence and we need first the definition of a Haar system of partitions of $\N$.

\begin{definition}{\label{Haar_system}}
A collection $(A^{n}_{j})_{j=0,n\in\N}^{2^{n}-1}$ of infinite subsets of $\N$ is said to be a Haar system of partitions of $\N$ if the following two conditions are satisfied:
\begin{enumerate}
\item For every $n\in\N$, we have that $\cup_{j=0}^{2^{n}-1}A^{n}_{j}=\N$ and $A^{n}_{j}\cap A^{n}_{j'}=\emptyset$ if $j\neq j'$.
\item For every $n\in\N$ and for every $0\leq j\leq 2^{n}-1$, we have that $A^{n}_{j}= A^{n+1}_{2j}\cup A^{n+1}_{2j+1}$.
\end{enumerate}
\end{definition}

The connection between the Haar-$\ell_{1}^{+}$ condition (defined below) and Riemann integrability is due to the fact that the infinite set $A^{n}_{j}$ of the Haar system of partitions of $\N$ corresponds to the dyadic rational numbers in the dyadic interval $\left(\frac{j}{2^{n}},\frac{j+1}{2^{n}}\right)$ (see, e.g., \cite[Theorem 3.1]{gaebler:sari:2023}).

\begin{definition}{\label{Haar-l1+}}
A normalized basic sequence $(x_{i})$ in $X$ is said to be Haar-$\ell_{1}^{+}$ if, for every Haar system of partitions $(A^{n}_{j})_{j=0,n\in\N}^{2^{n}-1}$ of $\N$, there exists a constant $C>0$ such that
\bes C\leq \lim_{n\to\infty}\sup\left\{ \frac{1}{2^{m}}\left\Vert \sum_{j=0}^{2^{m}-1}x_{i_{j}}\right\Vert \;\middle\vert\; m\geq n \text{ and } 2^{m}\leq i_{j}\in A^{m}_{j}\right\}.\ees
\end{definition}

\begin{remark}
\label{lebesgue on binary trees}
For a Banach space $X$ the following are equivalent.
\begin{enumerate}
    \item $X$ has the Lebesgue property.
    \item Every normalized basic sequence in $X$ is Haar-$\ell_{1}^{+}$.
    \item {For every collection $(x_\lambda)_{\lambda\in\mathcal{D}}$ of norm-one vectors in $X$, there exists a constant $\theta>0$ such that, for every $n\in\mathbb{N}$, there exists a subset $V\subset\{0,1\}^{n}$ such that
\bes\theta\leq \frac{1}{2^{n}}\left\Vert\sum_{\lambda\in V} x_{\mu_{\lambda}}\right\Vert \ees
for some choice of nodes $(\mu_{\lambda})_{\lambda\in V}$ with $\mu_{\lambda}\geq\lambda$.}
\end{enumerate}

Indeed, the assertion $(1)\iff (2)$ is established in \cite{gaebler:sari:2023}. In addition, it is easy to see that $(1)\implies (3)$ as collection witnessing the negation of (3) can be easily used to construct a Dirichlet-type function with Riemann integral zero. Lastly, for $(3)\implies (2)$, if there exists a normalized basic sequence $(x_i)$ in $X$ that is not Haar-$\ell_{1}^{+}$, then it is easy to construct a collection $(x_{\lambda})_{\lambda\in\mathcal{D}}$ of norm-one vectors for which the claim of $(3)$ does not hold as follows. Fix an order preserving bijection $\mathcal{D}\to \{A^n_j\mid n\in\N, 0\leq j\leq 2^n-1\}$ such that, for every $n\in\N$, $\lambda\mapsto A^{|\lambda|}_{j_\lambda}$, for some $0\leq j_\lambda\leq 2^{|\lambda|}-1$. Then, for every $\lambda\in\mathcal{D}$, choose $i_\lambda\in A^{|\lambda|}_{j_\lambda}$ such that $i_\lambda\geq 2^{|\lambda|}$ and put $x_\lambda = x_{i_\lambda}$. Note that $(x_\lambda)_{\lambda\in\mathcal{D}}$ can be chosen Schauder basic in the lexicographical oder of $\mathcal{D}$.
\end{remark}

\section{Construction of $X_\mathcal{D}$ and uniformly unique $\ell_{1}$ spreading model}{\label{sec3}}

In this section we define the norming set $W$ inducing the Banach space $X_\mathcal{D}$. It is defined as an increasing union $W = \cup_{k=0}^\infty W_k$, where $W_0 = \{\pm e_i^*\}_{i=1}^\infty$ and, in each step, having defined $W_k$, we formulate the operations and constraints that are used to produce $W_{k+1}$. The quantification of the metric constraints discussed in the introduction is defined inductively and carries a certain complexity; this is necessary to balance the uniform uniqueness of $\ell_1$ spreading models with the $\ell_\infty$ estimates on tree-indexed collection.  At the end of this section, we prove that every normalized weakly null sequence in $X_\mathcal{D}$ has a subsequence generating a spreading model that is 102-equivalent to the unit vector basis of $\ell_1$.

Fix a pair of strictly increasing sequences $(m_{j})$ and $(n_{j})$ of positive integers such that $m_{1}=100$, $n_{1}=1$,
\begin{enumerate}[label=(\arabic*)]
\item\label{condition m_j} $m_{j+1}\geq j^2m_j^2$, for every $j\in\N$,
\item\label{condition n_j} $n_{j+1}>\frac{4}{\log(100)}n_{j}\log(m_{j+1})$.
\end{enumerate}

Let $\mathcal{D} = \cup_{n=0}^\infty\{0,1\}^n$ with the partial order and metric $d$ defined in the introduction. We note that the completion $\overline{\mathcal{D}}$ of $\mathcal{D}$ is a compact metric space. Denote $\mathcal{M} = \{\infty\}\cup\{m_j:j\in\N\}$ and let $\phi:\mathcal{M}\to\mathcal{D}$ be a bijection so that the nodes of $\mathcal{D}$ are ordered lexicographically with respect to $\mathcal{M}$, that is, $\phi(\infty) = \emptyset$, $\phi(m_{1}) = (0)$, $\phi(m_{2})=(1)$, $\phi(m_{3}) = (0,0)$, and so on. Lastly, as a matter of notational convenience, we always write $d(m_{j},\cdot)$ instead of the formally correct $d(\phi(m_{j}),\cdot)$.

Now, let $W_{0}=\{\pm e_{i}^{*}\}_{i=1}^{\infty}\subset c_{00}$ and write $w(\pm e^{*}_{i})=m(\pm e^{*}_{i})=\infty$ for each $i\in\N$. Then, define the set
\bes W_{1}=W_{0}\cup\left\{\frac{1}{m_{j_{1}}\cdots m_{j_{l}}}\sum_{k=1}^{d}e^{*}_{i_{k}} \;\middle\vert\; \substack{ e^{*}_{i_{1}}<\ldots<e^{*}_{i_{d}} \text{ are } \\ S_{n_{j_{1}}+\cdots+n_{j_{l}}}\text{-admissible}} \right\} \ees
where $w(f)=m_{j_{1}}\cdots m_{j_{l}}$, $\vec w(f) = (m_{j_1},\ldots,m_{j_\ell})$, and $m(f)=\max_{1\leq z\leq l}m_{j_{z}}$ for $f\in W_{1}\setminus W_{0}$. { For $f$ in $W_1\setminus W_0$, $\vec w(f)$ may not be uniquely defined, and thus neither are $w(f)$ and $m(f)$. The notation $\vec w(f) = (m_{j_1},\ldots,m_{j_\ell})$ means that $f$ admits a representation as above, and it is fixed while working with $f$. This will also apply to $f\in W_k$, $k>1$. This is common in this type of construction and does not pose a problem. In fact, insisting on uniquely defining the weight may interfere with constructing functionals that are used to prove the properties of the space.} We next define the following notion of proximity with respect to $\overline{\mathcal{D}}$ for functionals in $W_{1}$.

\begin{definition}{\label{packed}}
The functionals $f_{1}<\ldots<f_{d}$ in $W_{1}$ are said to be $(n,\lambda)$-packed for some $\lambda\in\overline{\mathcal{D}}$ if the following conditions are satisfied:
\begin{enumerate}
\item $f_{1}<\ldots<f_{d}$ are $S_{n}$-admissible
\item $w(f_{i})\geq 2^{\max\supp(f_{i-1})}$ for $2\leq i\leq d$
\item $d(m(f_{i}),\lambda)\leq 2^{-\max\supp(f_{i-1})}$ for $2\leq i\leq d$
\end{enumerate}
and if, in addition, we have that $\max_{1\leq i\leq d}d(m(f_{i}),\lambda)\leq \eta$, then $f_{1}<\ldots<f_{d}$ are said to be $(n,\lambda,\eta)$-packed.
\end{definition}

This definition allows for the following inductive definition as well.

\begin{definition}{\label{big_packed}}
The functionals $f_{1}<\ldots<f_{d}$ in $W_{1}$ are said to be $(n_{j_{1}},\ldots,n_{j_{l}},\lambda)$-packed for some $\lambda\in\overline{\mathcal{D}}$ if the following claims hold:
\begin{enumerate}
\item $w(f_{i})\geq 2^{\max\supp(f_{i-1})}$ for $2\leq i\leq d$
\item There exists a partition $F_{1}<\ldots<F_{M}$ of $\{1,\ldots,d\}$ such that the set $\{f_{\min(F_{k})}\}_{k=1}^{M}$ is $S_{n_{j_{l}}}$-admissible
\item There exist nodes $\lambda_{1},\ldots,\lambda_{M}\in\overline{\mathcal{D}}$ such that:
\begin{enumerate}
\item $(f_{i})_{i\in F_{1}}$ is $(n_{j_{1}},\ldots,n_{j_{l-1}},\lambda_{1})$-packed
\item $(f_{i})_{i\in F_{k}}$ is $(n_{j_{1}},\ldots,n_{j_{l-1}},\lambda_{k},2^{-\max\supp(f_{\max(F_{k-1})})})$-packed for $2\leq k\leq M$
\item $d(\lambda_{k},\lambda)\leq 2^{-\max\supp(f_{\max(F_{k-1})})}$ for $2\leq k\leq M$
\end{enumerate}
\end{enumerate}
and if, in addition, we have that $\max_{1\leq i\leq d}d(m(f_{i}),\lambda)\leq\eta$, then we say that $f_{1}<\ldots<f_{d}$ are $(n_{j_{1}},\ldots,n_{j_{l}},\lambda,\eta)$-packed.
\end{definition}

We now define the next set of weighted functionals
\bes W_{2}=W_{1}\cup\left\{ \frac{1}{m_{j_{1}}\cdots m_{j_{l}}}\sum_{i=1}^{d}f_{i} \;\middle\vert\; \substack{ f_{1}<\ldots<f_{d} \text{ are in } W_{1} \text{ and are} \\  (n_{j_{1}},\ldots,n_{j_{l}},\lambda)\text{-packed for some } \lambda\in\overline{\mathcal{D}}} \right\} \ees
with $w(f)=m_{j_{1}}\cdots m_{j_{l}}$, $\vec w(f) = (m_{j_1},\ldots, m_{j_\ell})$, and $m(f)=\max_{1\leq z\leq l}m_{j_{z}}$. Now, once we have defined the sets $W_{0},W_{1},\ldots,W_{k}$, we replace $W_{1}$ in Definitions \ref{packed} and \ref{big_packed} with $W_{k}$ and we write that
\bes W_{k+1}=W_{k}\cup\left\{ \frac{1}{m_{j_{1}}\cdots m_{j_{l}}}\sum_{i=1}^{d}f_{i} \;\middle\vert\; \substack{ f_{1}<\ldots<f_{d} \text{ are in } W_{k} \text{ and are} \\  (n_{j_{1}},\ldots,n_{j_{l}},\lambda)\text{-packed for some } \lambda\in\overline{\mathcal{D}}} \right\} \ees
for each $k\in\N$. We then define our norming set $W=\cup_{k=0}^{\infty}W_{k}$ and we let the Banach space $X_\mathcal{D}$ be the completion of $c_{00}$ with respect to the norm that is given by $\|x\|=\sup\{f(x) \mid f\in W\}$ for $x\in c_{00}$.

\begin{remark}
\label{unconditionality and the like}
By induction on $k$, where $W = \cup_{k=0}^\infty W_k$, it is easy to prove the following.
\begin{enumerate}[label=(\arabic*)]
    \item\label{unconditionality and the like 1} For every $(n_{j_1},\ldots,n_{j_\ell},\lambda)$-packed sequence $f_1<\cdots<f_d$ in $W$,
    \[f = \frac{1}{m_{j_1}\cdots m_{j_\ell}}\sum_{i=1}^df_i\]
    is in $W$.
    \end{enumerate}
Incorporating a nested induction on $\ell$, the following items can be shown.
    \begin{enumerate}[resume,label=(\arabic*)]

    \item\label{unconditionality and the like 2}  A subsequence of a $(n_{j_1},\ldots,n_{j_\ell},\lambda)$-packed sequence $f_1<\cdots<f_d$ in $W$ is also $(n_{j_1},\ldots,n_{j_\ell},\lambda)$-packed.

    \item\label{unconditionality and the like 3} If $f_1<\cdots<f_d$ is a $(n_{j_1},\ldots,n_{j_\ell},\lambda)$-packed sequence in $W$ and $g_1<\cdots<g_d$ are in $c_{00}$ such that, for $1\leq i\leq d$ and $s\in\mathrm{supp}(g_i)$, we have $|g_i(e_s)| = |f_i(e_s)|$ then $g_1$,\ldots,$g_d$ are also in $W$ and they are $(n_{j_1},\ldots,n_{j_\ell},\lambda)$-{packed}.

\end{enumerate}
In particular, \ref{unconditionality and the like 2} and \ref{unconditionality and the like 3} yield that $W$ is closed under restricting the supports of its members to subsets and changing the signs of their scalar coefficients, and thus, $(e_i)$ is a 1-unconditional basis for $X_\mathcal{D}$.
\end{remark}

We now prove that every spreading model of $X_\mathcal{D}$ is uniformly equivalent to the unit vector basis of $\ell_{1}$. In fact, the precise statement of the following theorem is slightly stronger than this assertion; it will be required in its full generality later, in the proof of Proposition \ref{scc existene in block subspace}.

\begin{theorem}{\label{l1_spreading_model}}
Let $(x_{r})$ be a normalized block  sequence  of $(e_{i})$. Then, there exists an infinite subset $\Omega\subset\N$ such that, for every $j_{0}\in\N$, there is the estimate
\bes \left\Vert \sum_{r\in F}c_{r}x_{r}\right\Vert \geq 0.99\frac{1}{m_{j_{0}}}\sum_{r\in \Gamma}|c_{r}| \ees
for every subset $\Gamma\subset\Omega$ such that $(x_{r})_{r\in \Gamma}$ are $S_{n_{j_{0}}}$-admissible and for every choice $(c_{r})_{r\in \Gamma}$ of scalars.
\end{theorem}
\begin{proof}
Choose for every $r\in\N$ a norming functional $f_{r}\in W$ so that $f_{r}(x_{r})=1$ and so that $\supp(f_{r})\subset\supp(x_{r})$. Then, there are two cases to consider.
\begin{itemize}
\item[\textbf{(i)}] $\limsup_{r\to\infty}w(f_{r})=L<\infty$
\end{itemize} 
Here, we may pass to an infinite subset $\Omega_{1}\subset\N$ so that $w(f_{r})=m_{j_{1}}\cdots m_{j_{l}}$ for each $r\in\Omega_{1}$. Note that each $f_{r}$ with $r\in\Omega_{1}$ is then $(n_{j_{1}},\ldots,n_{j_{l}},\lambda_{r})$-packed for some $\lambda_{r}\in\overline{\mathcal{D}}$. We may then pass to a further infinite subset $\Omega_{2}\subset\Omega_{1}$ so that $d(\lambda_{r},\lambda)\leq \frac{1}{2}\cdot 2^{-\max\supp(x_{r-1})}$ for each $r\in\Omega_{2}\setminus\{\min(\Omega_{2})\}$ and for some $\lambda\in\overline{\mathcal{D}}$. Now, let $\Omega=\Omega_{2}$ and consider the subsequence $(x_{r})_{r\in\Omega_{2}}$. Let $\Gamma\subset\Omega$ be a subset such that the block vectors $(x_{r})_{r\in \Gamma}$ are $S_{n_{j_{0}}}$-admissible and let 
\bes f_{r}=\frac{1}{m_{j_{1}}\cdots m_{j_{l}}}\sum_{q=1}^{d_{r}}f^{r}_{q} \qquad r\in\Gamma \ees
be the associated norming functionals. For each $r\in\Gamma$, let $F^{r}_{1}<\ldots<F^{r}_{M_{r}}$ be the partition of $\{1,\ldots,d_{r}\}$ and let $\lambda^{r}_{1},\ldots,\lambda^{r}_{M_{r}}\in\overline{\mathcal{D}}$ be the corresponding nodes so that Definition \ref{big_packed} is satisfied because the functionals $f^{r}_{1}<\ldots<f^{r}_{d_{r}}$ are $(n_{j_{1}},\ldots,n_{j_{l}},\lambda_{r})$-packed. We will now show that the concatenation
\bes f^{\min(\Gamma)}_{\min(F^{\min(\Gamma)}_{2})}<\ldots<f^{\min(\Gamma)}_{d_{\min(\Gamma)}}<\ldots<f^{\max(\Gamma)}_{\min(F_{2}^{\max(\Gamma)})}<\ldots<f^{\max(\Gamma)}_{d_{\max(\Gamma)}} \ees
is in $W$. First, note that there is some $\kappa\in\N$ so that each functional in this concatenation is a member of $W_{\kappa}$. It therefore remains to show that these functionals are $(n_{j_{1}},\ldots,n_{j_{l}},n_{j_{0}},\lambda)$-packed. Indeed, it follows immediately that they are $S_{n_{j_{1}}+\cdots+n_{j_{l}}+n_{j_{0}}}$-admissible and are very fast growing in the sense of Definition \ref{big_packed}. Then, let $F_{\min(\Gamma)}<\ldots<F_{\max(\Gamma)}$ be the partition of the number of functionals in the above concatenation such that $(f_{i})_{i\in F_{r}}=\{f^{r}_{\min(F^{r}_{2})},\ldots,f^{r}_{d_{r}}\}$ and note that $\{f_{\min(F_{r})}\}_{r\in\Gamma}=\{f^{r}_{\min(F^{r}_{2})}\}_{r\in\Gamma}$ is $S_{n_{j_{0}}}$-admissible. Moreover:
\begin{enumerate}
\item $(f_{i})_{i\in F_{1}}=\{f^{1}_{\min(F^{1}_{2})},\ldots,f^{1}_{d_{1}}\}$ is $(n_{j_{1}},\ldots,n_{j_{l}},\lambda_{1})$-packed.
\item $(f_{i})_{i\in F_{r}}=\{f^{r}_{\min(F^{r}_{2})},\ldots,f^{r}_{d_{r}}\}$ is $(n_{j_{1}},\ldots,n_{j_{l}},\lambda_{r})$-packed for each $r\in\Gamma\setminus\{\min(\Gamma)\}$ and, in addition, for each $f_{i}$ with $i\in F^{r}_{k}$ (where $2\leq k\leq M_{r}$), there is
\begin{align*} d(m(f_{i}),\lambda_{r}) &\leq d(m(f_{i}),\lambda^{r}_{k})+d(\lambda^{r}_{k},\lambda_{r}) \\ &\leq 2\cdot  2^{-\max\supp(f_{\max(F^{r}_{k-1})})}\leq 2^{-\max\supp(f_{\max(F_{r-1})})}\end{align*}
so that, in fact, $(f_{i})_{i\in F_{r}}$ is $(n_{j_{1}},\ldots,n_{j_{l}},\lambda_{r}, 2^{-\max\supp(f_{\max(F_{r-1})})})$-packed.
\item $d(\lambda_{r},\lambda)\leq\frac{1}{2}\cdot 2^{-\max\supp(x_{r-1})}\leq 2^{-\max\supp(f_{\max(F_{r-1})})}$ for each $r\in\Gamma\setminus\{\min(\Gamma)\}$.
\end{enumerate}
Putting these pieces together, it follows that the functional
\bes \mathscr{F}=\frac{1}{m_{j_{1}}\cdots m_{j_{l}}m_{j_{0}}}\sum_{r\in\Gamma}\sum_{q=\min(F^{r}_{2})}^{d_{r}}f^{r}_{q} \ees
is $(n_{j_{1}},\ldots,n_{j_{l}},n_{j_{0}},\lambda)$-packed and is therefore a member of $W_{\kappa+1}\subset W$. Now, because $(x_{r})_{r\in\Omega}$ is itself $1$-unconditional, there is the estimate
\begin{align*} \left\Vert \sum_{r\in \Gamma}c_{r}x_{r}\right\Vert = \left\Vert \sum_{r\in\Gamma}|c_{r}|x_{r}\right\Vert &\geq\mathscr{F}\left(\sum_{r\in\Gamma}|c_{r}|x_{r}\right) \\ &=\frac{1}{m_{j_{1}}\cdots m_{j_{l}}m_{j_{0}}}\sum_{r\in\Gamma}\sum_{q=\min(F^{r}_{2})}^{d_{r}}f^{r}_{q}\left(\sum_{p\in\Gamma}|c_{p}|x_{p}\right) \\ &=\frac{1}{m_{j_{0}}}\sum_{r\in\Gamma}\frac{|c_{r}|}{m_{j_{1}}\cdots m_{j_{l}}}\sum_{q=\min(F^{r}_{2})}^{d_{r}}f^{r}_{q}(x_{r}) \\ &=\frac{1}{m_{j_{0}}}\sum_{r\in\Gamma}|c_{r}|\left(f_{r}(x_{r})-\frac{1}{m_{j_{1}}\cdots m_{j_{l}}}\sum_{q\in F^{r}_{1}}f^{r}_{q}(x_{r})\right) \\ &\geq \frac{1}{m_{j_{0}}}\sum_{r\in\Gamma}|c_{r}|\left(1 - \frac{1}{m_{j_{l}}}\cdot\frac{1}{m_{j_{1}}\cdots m_{j_{l-1}}}\sum_{q\in F^{r}_{1}}f^{r}_{q}(x_{r})\right) \\ &\geq \frac{1}{m_{j_{0}}}\sum_{r\in\Gamma}|c_{r}|\left(1-\frac{1}{m_{1}}\right) = \frac{m_{1}-1}{m_{1}m_{j_{0}}}\sum_{r\in\Gamma}|c_{r}|\end{align*}
where we have used the fact that $\frac{1}{m_{j_{1}}\cdots m_{j_{l-1}}}\sum_{q\in F^{r}_{1}}f^{r}_{q}\in W$, and this completes the proof for the first case.

\begin{itemize}
\item[\textbf{(ii)}] $\limsup_{r\to\infty}w(f_{r})=\infty$
\end{itemize}
In this case, there exists an infinite subset $\Omega_{1}\subset\N$ such that $w(f_{r})$ for $r\in\Omega_{1}\setminus\{\min(\Omega_{1})\}$ is very fast growing in the sense of Definition \ref{packed}. Then, let $\lambda_{r}$ be the node in $\overline{\mathcal{D}}$ that corresponds to $m(f_{r})$ for each $r\in\Omega_{1}$. We may then pass to a further infinite subset $\Omega_{2}\subset\Omega_{1}$ so that $d(\lambda_{r},\lambda)\leq \frac{1}{2}\cdot 2^{-\max\supp(x_{r-1})}$ for some $\lambda\in\overline{\mathcal{D}}$ and for each $r\in\Omega_{2}\setminus\{\min(\Omega_{2})\}$. Now, let $\Omega=\Omega_{2}$ and consider the subsequence $(x_{r})_{r\in\Omega}$. Let $\Gamma\subset\Omega$ be a subset such that the block vectors $(x_{r})_{r\in\Gamma}$ are $S_{n_{j_{0}}}$-admissible and let $f_{\min(\Gamma)}<\ldots<f_{\max(\Gamma)}$ be the corresponding norming functionals. Note that there is some $\kappa\in\N$ so that $f_{r}\in W_{\kappa}$ for each $r\in\Gamma$. Note also that, by construction, these norming functionals are $S_{n_{j_{0}}}$-admissible and are very fast growing in the sense of Definition \ref{packed}. Moreover, there is the estimate
\bes d(\lambda_{r},\lambda)\leq 2^{-\max\supp(x_{r-1})}\leq 2^{-\max\supp(f_{r-1})} \ees
for each $r\in\Gamma\setminus\{\min(\Gamma)\}$ so that these functionals are $(n_{j_{0}},\lambda)$-packed. Then, we have that $\frac{1}{m_{j_{0}}}\sum_{r\in\Gamma}f_{r}\in W_{\kappa+1}\subset W$ so that
\begin{align*} \left\Vert \sum_{r\in\Gamma}c_{r}x_{r}\right\Vert =\left\Vert \sum_{r\in\Gamma}|c_{r}|x_{r}\right\Vert &\geq \frac{1}{m_{j_{0}}}\sum_{r\in\Gamma}f_{r}\left(\sum_{p\in\Gamma}|c_{p}|x_{p}\right) \\ &=\frac{1}{m_{j_{0}}}\sum_{r\in\Gamma}|c_{r}|f_{r}(x_{r})=\frac{1}{m_{j_{0}}}\sum_{r\in\Gamma}|c_{r}|\geq\frac{m_{1}-1}{m_{1}m_{j_{0}}}\sum_{r\in\Gamma}|c_{r}| \end{align*}
and this completes the proof of the theorem.
\end{proof}

\section{Auxiliary Banach spaces and $c_{0}$-type estimate}{\label{sec4}}

So far we have proved that all spreading models admitted by normalized weakly null sequences in $X_\mathcal{D}$ are uniformly equivalent to the unit vector basis of $\ell_1$. The remaining goal is to define in each subspace of $X_\mathcal{D}$ a normalized collection $(x_\lambda)_{\lambda\in\mathcal{D}}$ satisfying the $\ell_\infty$-estimate \eqref{intro ellinfty estimate}, and this will require some effort. As a first step, we define an auxiliary norming set $W^{\text{aux},\delta}$ inducing a norm on $c_{00}$ and an auxiliary Banach space $X_{\mathcal{D}}^{\text{aux},\delta}$. We then prove estimates for the norm of linear combinations of ``simple'' vectors, specifically weighted $(n,\varepsilon)$ basic special convex combinations, in the auxiliary space. These will be relevant later when a tool called the basic inequality will yield high-precision estimates of linear combinations of special sequences of vectors, called rapidly increasing sequences, in $X_{\mathcal{D}}$, in terms of sequences of basis vectors in $X_{\mathcal{D}}^{\text{aux},\delta}$.

We derive in this section a $c_{0}$-type estimate for the action of auxiliary functionals on weighted $(n,\ep)$-basic scc vectors. The idea is that the functional evaluates to something small if its weight is in some sense a mismatch to the $(n,\ep)$-basic scc vector weight. Let us first define the auxiliary functionals and auxiliary Banach spaces.

We will define an auxiliary Banach space for each $\delta>0$. First, fix $\delta>0$ and let $W_{0}^{\text{aux},\delta}=\{\pm e_{i}^{*}\}_{i=1}^{\infty}\subset c_{00}$ as before. Then, also as before, define $w(\pm e^{*}_{i})=m(\pm e^{*}_{i})=\infty$ and define the set of auxiliary weighted functionals
\bes W^{\text{aux},\delta}_{1}=W^{\text{aux},\delta}_{0}\cup\left\{\frac{1}{m_{j_{1}}\cdots m_{j_{l}}}\sum_{k=1}^{d}e^{*}_{i_{k}} \;\middle\vert\; \substack{ e^{*}_{i_{1}}<\ldots<e^{*}_{i_{d}} \text{ are } \\ S_{n_{j_{1}}+\cdots+n_{j_{l}}}\ast\mathcal{A}_{3}\text{-admissible}} \right\} \ees
where $w(f)=m_{j_{1}}\cdots m_{j_{l}}$, $m(f)=\max_{1\leq z\leq l}m_{j_{z}}$, and $\vec w(f) = (m_{j_1},\ldots,m_{j_\ell})$, for $f\in W_{1}^{\text{aux},\delta}\setminus W_{0}^{\text{aux},\delta}$. We next define the following notion of proximity with respect to $\overline{\mathcal{D}}$ for functionals in $W^{\text{aux},\delta}_{1}$.

\begin{definition}{\label{delta_packed}}
The functionals $f_{1}<\ldots<f_{d}$ in $W^{\text{aux},\delta}_{1}$ are said to be $(n,\lambda)_{\delta}$-packed for some $\lambda\in\overline{\mathcal{D}}$ if the following conditions are satisfied:
\begin{enumerate}
\item $f_{1}<\ldots<f_{d}$ are $S_{n}\ast\mathcal{A}_{3}$-admissible
\item $w(f_{i})\geq 2^{\max\supp(f_{i-2})}$ for $3\leq i\leq d$
\item $d(m(f_{i}),\lambda)\leq\delta$ for $f_{i}\notin W^{\text{aux},\delta}_{0}$ and for $2\leq i\leq d$
\end{enumerate}
\end{definition}

This definition is the base case for the following inductive definition.

\begin{definition}{\label{big_delta_packed}}
The functionals $f_{1}<\ldots<f_{d}$ in $W^{\text{aux},\delta}_{1}$ are said to be $(n_{j_{1}},\ldots,n_{j_{l}},\lambda)_{\delta}$-packed for some $\lambda\in\overline{\mathcal{D}}$ if the following conditions hold:
\begin{enumerate}
\item $w(f_{i})\geq 2^{\max\supp(f_{i-2})}$ for $3\leq i\leq d$
\item There exists a partition $F_{1}<\ldots<F_{M}$ of $\{1,\ldots,d\}$ such that the set $\{f_{\min(F_{k})}\}_{k=1}^{M}$ is $S_{n_{j_{l}}}$-admissible and such that the functionals $(f_{i})_{i\in F_{k}}$ are $S_{n_{j_{1}}+\cdots+n_{j_{l-1}}}\ast\mathcal{A}_{3}$-admissible for $1\leq k\leq M$
\item There exists a node $\lambda_{1}\in\overline{\mathcal{D}}$ such that $(f_{i})_{i\in F_{1}}$ is $(n_{j_{1}},\ldots,n_{j_{l-1}},\lambda_{1})_{\delta}$-packed
\item $d(m(f_{i}),\lambda)\leq\delta$ for $f_{i}\in\left(\bigcup_{k=2}^{M}F_{k}\right)\setminus W^{\text{aux},\delta}_{0}$
\end{enumerate}
\end{definition}

It is worth noting that a subset of $(n_{j_{1}},\ldots,n_{j_{l}},\lambda)_{\delta}$-packed functionals in $W^{\text{aux},\delta}_{1}$ is still $(n_{j_{1}},\ldots,n_{j_{l}},\lambda)_{\delta}$-packed, and we now define the next set of auxiliary weighted functionals
\bes W^{\text{aux},\delta}_{2}=W^{\text{aux},\delta}_{1}\cup\left\{ \frac{1}{m_{j_{1}}\cdots m_{j_{l}}}\sum_{i=1}^{d}f_{i} \;\middle\vert\; \substack{ f_{1}<\ldots<f_{d} \text{ are in } W^{\text{aux},\delta}_{1} \text{ and are} \\  (n_{j_{1}},\ldots,n_{j_{l}},\lambda)_{\delta}\text{-packed for some } \lambda\in\overline{\mathcal{D}}} \right\} \ees
with $w(f)=m_{j_{1}}\cdots m_{j_{l}}$, $m(f)=\max_{1\leq z\leq l}m_{j_{z}}$, and $\vec w(f) = (m_{j_1},\ldots,m_{j_\ell})$. Now, once we have defined the sets $W^{\text{aux},\delta}_{0},W^{\text{aux},\delta}_{1},\ldots,W^{\text{aux},\delta}_{k}$, we replace $W^{\text{aux},\delta}_{1}$ in Definitions \ref{delta_packed} and \ref{big_delta_packed} with $W^{\text{aux},\delta}_{k}$ and we write that
\bes W^{\text{aux},\delta}_{k+1}=W^{\text{aux},\delta}_{k}\cup\left\{ \frac{1}{m_{j_{1}}\cdots m_{j_{l}}}\sum_{i=1}^{d}f_{i} \;\middle\vert\; \substack{ f_{1}<\ldots<f_{d} \text{ are in } W^{\text{aux},\delta}_{k} \text{ and are} \\  (n_{j_{1}},\ldots,n_{j_{l}},\lambda)_{\delta}\text{-packed for some } \lambda\in\overline{\mathcal{D}}} \right\} \ees
for each $k\in\N$. We then define our norming set $W^{\text{aux},\delta}=\bigcup_{k=0}^{\infty}W^{\text{aux},\delta}_{k}$ and we let the Banach space $X_{\mathcal{D}}^{\text{aux},\delta}$ be the completion of $c_{00}$ with respect to the norm that is given by $\|x\|_{W^{\text{aux},\delta}}=\sup\{f(x) \mid f\in W^{\text{aux},\delta}\}$ for $x\in c_{00}$. We will need the following lemma before we can prove the main result of this section.

\begin{lemma}{\label{aux_lemma}}
Let $f=\frac{1}{m_{j_{1}}\cdots m_{j_{l}}}\sum_{i=1}^{d}f_{i}\in W^{\text{aux},\delta}$. Then, $\{ i \mid f_{i}\notin W^{\text{aux},\delta}_{0}\}=E_{0}\cup E_{1}\cup\ldots\cup E_{l}$ where $\mathrm{card}(E_0)\leq 1$ and, for each $1\leq z\leq l$, $(f_{i})_{i\in E_{z}}$ is $S_{n_{j_{1}}+\cdots+n_{j_{z}}}\ast\mathcal{A}_{3}$-admissible and $\max_{i\in E_{z}}d_{T}(m(f_{i}),\lambda_{z})\leq\delta$ for some $\lambda_{z}\in\overline{\mathcal{D}}$.
\end{lemma}
\begin{proof}
We proceed by induction on the length $l$ of the tuple $(m_{j_{1}},\ldots,m_{j_{l}})$. If $l=1$, then $f=\frac{1}{m_{j_{1}}}\sum_{i=1}^{d}f_{i}\in W^{\text{aux},\delta}_{k}$ for some $k\geq 1$ and we put 
\begin{align*}  E_{0}&=\{1\}\setminus\{ i \mid f_{i}\in W^{\text{aux},\delta}_{0}\} \\ E_{1}&=\{2,\ldots,d\}\setminus \{i \mid f_{i}\in W^{\text{aux},\delta}_{0}\} \end{align*}
and note that, if $k=1$, then $E_{0}=E_{1}=\emptyset$ so that the result holds vacuously. Assume then that $k> 1$ and note that $f_{1}<\ldots<f_{d}$ are in $W^{\text{aux},\delta}_{k-1}$ and are $(n_{j_{1}},\lambda)_{\delta}$-packed for some $\lambda\in\overline{\mathcal{D}}$. It therefore follows by definition that $(f_{i})_{i\in E_{1}}$ are $S_{n_{j_{1}}}\ast\mathcal{A}_{3}$-admissible and that $\max_{i\in E_{1}}d(m(f_{i}),\lambda)\leq \delta$.

Assume now that the result holds for functionals of the form $f=\frac{1}{m_{j_{1}}\cdots m_{j_{l}}}\sum_{i=1}^{d}f_{i}\in W^{\text{aux},\delta}$ and let $g=\frac{1}{m_{j_{1}}\cdots m_{j_{l}}m_{j_{l+1}}}\sum_{i=1}^{d}g_{i}\in W^{\text{aux},\delta}$. As above, the result holds vacuously if $g\in W^{\text{aux},\delta}_{1}$, so assume that $g\in W^{\text{aux},\delta}_{k}$ for some $k>1$. In this case, we have that $g_{1}<\ldots<g_{d}$ are in $W^{\text{aux},\delta}_{k-1}$ and are $(n_{j_{1}},\ldots,n_{j_{l}},n_{j_{l+1}},\lambda)_{\delta}$-packed for some $\lambda\in\overline{\mathcal{D}}$. It follows that there is a partition $F_{1}<\ldots<F_{M}$ of $\{1,\ldots,d\}$ as in Definition \ref{big_delta_packed}. Note that in particular that $(g_{i})_{i\in F_{1}}$ are $(n_{j_{1}},\ldots,n_{j_{l}},\lambda_{1})_{\delta}$-packed for some $\lambda_{1}\in\overline{\mathcal{D}}$ so that the functional $\tilde{g}=\frac{1}{m_{j_{1}}\cdots m_{j_{l}}}\sum_{i\in F_{1}}g_{i}\in W^{\text{aux},\delta}$. Then, by the inductive hypothesis, there are sets $E_{0}\cup E_{1}\cup\ldots\cup E_{l}\subset F_{1}$ so that the required claim holds for $\tilde{g}$. Now, let $E_{l+1}=\left(\bigcup_{r=2}^{M}F_{r}\right)\setminus\{ i \mid g_{i}\in W^{\text{aux},\delta}_{0}\}$ and it follows that the required claim holds for $g$ itself so that the proof is complete.
\end{proof}

We now derive the aforementioned $c_{0}$-type estimate for auxiliary weighted functionals whose weights in some sense mismatch the weights of weighted $(n,\ep)$-basic scc vectors. First, we make a remark about the relation between the weight of a functional and the admissibility of its components.

\begin{remark}
\label{weight-to-admissibility condition}
Let $j\in\N$ and $f\in W$ such  $m(f)<m_j$ and $w(f) < m_j^4$, then, if $\vec w(f)= (m_{j_1},\ldots,m_{j_\ell})$, $n_{j_1}+\cdots +n_{j_\ell} < n_j$. Indeed, $m_{1}^{l}\leq w(f) <m_{j}^{4}$, and thus, $l<4\log(m_{j})/\log(100)$. By condition \ref{condition n_j} on page \pageref{condition n_j},
\[n_{j_{1}}+\cdots+n_{j_{l}}\leq ln_{j-1}< n_{j-1}4\frac{\log(m_{j})}{\log(100)}\leq n_{j}.\]
\end{remark}

\begin{proposition}{\label{c0_type_estimate}}
Let $N\in\N\cup\{0\}$ and choose $(t_{i})_{i=1}^{2^{N}}$ such that $d(m_{t_{i}},m_{t_{i'}})\geq 2^{-N}$ for each $i\neq i'$. Then, choose $\ep_{i}\in\left(0,\frac{1}{3m_{t_{i}}2^{N}}\right)$ and define a $(n_{t_{i}},\ep_{i})$-basic scc by
\bes x_{i}=m_{t_{i}}\sum_{r\in F_{i}}c_{r}^{i}e_{i_{r}} \ees
for each $1\leq i\leq 2^{N}$. It follows that $\left\vert f\left(\sum_{i\in\Omega}a_{i}x_{i}\right)\right\vert\leq \left(1+\frac{4}{\sqrt{w(f)}}\right)\max_{i\in\Omega}|a_{i}|$ for every subset $\Omega\subset\{1,\ldots,2^{N}\}$, for every choice $(a_{i})_{i\in\Omega}$ of scalars, and for every auxiliary weighted functional $f=\frac{1}{m_{j_{1}}\cdots m_{j_{l}}}\sum_{q=1}^{d}f_{q}\in W^{\text{aux},\delta}$ where $\delta<2^{-(N+1)}$.
\end{proposition}
\begin{proof}
Let $\Omega\subset\{1,\ldots,2^{N}\}$ be given. We will proceed by induction on the level $k$ norming sets $W^{\text{aux},\delta}_{k}$ and the inductive hypothesis will include that $\left\vert f\left(\sum_{i\in A_{f}(\Omega)}a_{i}x_{i}\right)\right\vert<\frac{4\gamma}{\sqrt{w(f)}}$ if $f\notin W_{0}^{\text{aux},\delta}$ where $\gamma=\max_{i\in\Omega}|a_{i}|$ and where $A_{f}(\Omega) :=\{i\in\Omega \mid m_{t_{i}}\neq m(f)\}$. The base case is the level $k=1$ norming set and we first let $f\in W^{\text{aux},\delta}_{0}$ so that   
\bes \left\vert f\left(\sum_{i\in\Omega}a_{i}x_{i}\right)\right\vert = \left\vert e^{*}_{i_{0}}\left(\sum_{i\in\Omega}a_{i}m_{t_{i}}\sum_{r\in F_{i}}c^{i}_{r}e_{r} \right)\right\vert \leq \max_{i\in\Omega}\ep_{i}m_{t_{i}}|a_{i}| < \max_{i\in\Omega}|a_{i}| \ees
for some $i_{0}\in\N$. It therefore remains to check functionals of the form
\bes f=\frac{1}{m_{j_{1}}\cdots m_{j_{l}}}\sum_{q=1}^{d}e^{*}_{i_{q}}\in W^{\text{aux},\delta}_{1}\setminus W^{\text{aux},\delta}_{0} \ees
and we write $A_{f}(\Omega)=\{ i\in\Omega \mid m_{t_{i}}\neq m(f)\}=A_{1}\cup A_{2}\cup A_{3}$ where
\begin{align*} A_{1}&=\{ i\in\Omega \mid m_{t_{i}}<m(f)\} \\ A_{2}&=\{ i\in\Omega \mid m(f)<m_{t_{i}}\text{ and } n_{j_{1}}+\cdots+n_{j_{l}}<n_{t_{i}}\} \\ A_{3}&=\{ i\in\Omega \mid m(f)<m_{t_{i}} \text{ and } n_{j_{1}}+\cdots+n_{j_{l}}\geq n_{t_{i}}\} \end{align*}
so that $\Omega\setminus A_{f}(\Omega)$ is at most a singleton $\{i_{0}\}$ with $m_{t_{i_{0}}}=m(f)$. Then, we have that
\bes \left\vert f\left(\sum_{i\in\Omega}a_{i}x_{i}\right)\right\vert\leq \sum_{p=1}^{3}\left\vert f\left(\sum_{i\in A_{p}}a_{i}x_{i}\right)\right\vert + |f(a_{i_{0}}x_{i_{0}})|\leq \sum_{p=1}^{3}\left\vert f\left(\sum_{i\in A_{p}}a_{i}x_{i}\right)\right\vert + \max_{i\in\Omega}|a_{i}| \ees
because $|f(x_{i_{0}})|\leq \frac{m_{t_{i_{0}}}}{w(f)}\leq 1$. Now, recall that $\gamma=\max_{i\in\Omega}|a_{i}|$ and note that, for $A_{1}$, there is the estimate
\begin{align*} \left\vert f\left(\sum_{i\in A_{1}}a_{i}x_{i}\right)\right\vert&\leq \sum_{i\in A_{1}}|f(a_{i}x_{i})|\leq \frac{\gamma}{w(f)}\sum_{i\in A_{1}}m_{t_{i}} \\ &\leq \frac{\gamma}{w(f)}\text{card}(A_{1})m_{t_{\max(A_{1})}} \leq \frac{\gamma}{w(f)}t_{\max(A_{1})}m_{t_{\max(A_{1})}} \\ &\leq \frac{\gamma}{w(f)}(\text{ind}(m(f))-1)m_{\text{ind}(m(f))-1} \leq \frac{\gamma m_{\text{ind}(m(f))}}{w(f)\sqrt{m_{\text{ind}(m(f))}}} \leq \frac{\gamma}{\sqrt{w(f)}}\end{align*}
where we have used the first condition on the sequence $(m_{j})$ and the fact that $t_{i}<\text{ind}(m(f)):=\max_{1\leq z\leq l}j_{z}$ for each $i\in A_{1}$. Next, for $A_{2}$, note that
\begin{align*} \left\vert f\left(\sum_{i\in A_{2}}a_{i}x_{i}\right)\right\vert &\leq \frac{\gamma}{w(f)}\sum_{i\in A_{2}}m_{t_{i}}\left\vert\sum_{q=1}^{d}e^{*}_{i_{q}}\left(\sum_{r\in F_{i}}c^{i}_{r}e_{r}\right) \right\vert \\ &\leq \frac{\gamma}{w(f)}\sum_{i\in A_{2}} 3m_{t_{i}}\ep_{i} \leq\frac{\gamma}{w(f)}\sum_{i\in A_{2}}\frac{1}{2^{N}}\leq\frac{\gamma}{w(f)}\leq\frac{\gamma}{\sqrt{w(f)}} \end{align*}
where we have used the observation that immediately precedes Proposition \ref{basic-scc-maximal} and the fact that $e^{*}_{i_{1}}<\ldots<e^{*}_{i_{d}}$ are $S_{n_{j_{1}}+\cdots+n_{j_{l}}}\ast\mathcal{A}_{3}$-admissible with $n_{j_{1}}+\cdots+n_{j_{l}}<n_{t_{i}}$ for each $i\in A_{2}$. Lastly, for $A_{3}$, note that 
\begin{align*} \left\vert f\left(\sum_{i\in A_{3}}a_{i}x_{i}\right)\right\vert &\leq \frac{\gamma}{w(f)}\sum_{i\in A_{3}}m_{t_{i}} \leq \frac{\gamma}{w(f)}|A_{3}|m_{t_{\max(A_{3})}} \\ &\leq\frac{\gamma m^{2}_{t_{\max(A_{3})}}}{w(f)}\leq \frac{\gamma\sqrt{w(f)}}{w(f)}=\frac{\gamma}{\sqrt{w(f)}}\end{align*}
because $m_{t_{i}}^{2}\leq\sqrt{w(f)}$ for each $i\in A_{3}$, by Remark \ref{weight-to-admissibility condition}. It now follows by putting these pieces together that
\bes \left\vert f\left(\sum_{i\in\Omega}a_{i}x_{i}\right)\right\vert \leq \gamma+\frac{3\gamma}{\sqrt{w(f)}}\ees
for each $f\in W^{\text{aux},\delta}_{1}$ and, in particular, we have that $\left\vert f\left(\sum_{i\in A_{f}(\Omega)}a_{i}x_{i}\right)\right\vert < \frac{4\gamma}{\sqrt{w(f)}}$ as is required as part of the inductive hypothesis.

Assume now that the statement of Proposition \ref{c0_type_estimate} holds for every functional $f\in W^{\text{aux},\delta}_{k}$, let $\Omega\subset\{1,\ldots,2^{N}\}$ be given, and let 
\bes f=\frac{1}{m_{j_{1}}\cdots m_{j_{l}}}\sum_{q=1}^{d}f_{q}\in W^{\text{aux},\delta}_{k+1}\setminus W^{\text{aux},\delta}_{k} \ees
where $f_{q}\in W^{\text{aux},\delta}_{k}$ for each $1\leq q\leq d$. Define the sets $B=\{q \mid f_{q}\in W^{\text{aux},\delta}_{0}\}$ and $C=\{3,\ldots,d\}\setminus B$. We may then write $f=\sum_{\alpha=1}^{4}g_{\alpha}$ where
\bes g_{1}=\frac{1}{w(f)}f_{1},\text{ }g_{2}=\frac{1}{w(f)}f_{2},\text{ }g_{3}=\frac{1}{w(f)}\sum_{q\in B}f_{q}, \text{ and } g_{4}=\frac{1}{w(f)}\sum_{q\in C}f_{q} \ees
and it follows that
\bes \left\vert f\left(\sum_{i\in\Omega}a_{i}x_{i}\right)\right\vert \leq \left\vert f\left(\sum_{i\in A_{f}(\Omega)}a_{i}x_{i}\right)\right\vert + |f(a_{i_{0}}x_{i_{0}})| \leq \max_{i\in\Omega}|a_{i}|+\left\vert f\left(\sum_{i\in A_{f}(\Omega)}a_{i}x_{i}\right)\right\vert \ees
where, as before, $A_{f}(\Omega)=\{ i\in\Omega \mid m_{t_{i}}\neq m(f)\}$ so that $\Omega\setminus A_{f}(\Omega)$ is at most a singleton $\{i_{0}\}$ with $m_{t_{i_{0}}}=m(f)$, in which case, $|f(x_{i_{0}})|\leq \frac{m_{t_{i_{0}}}}{w(f)}\leq 1$. Let $\gamma=\max_{i\in\Omega}|a_{i}|$ as before and it now remains to estimate
\begin{align*} \left\vert f\left(\sum_{i\in A_{f}(\Omega)}a_{i}x_{i}\right)\right\vert&\leq \sum_{\alpha=1}^{4}\left\vert g_{\alpha}\left(\sum_{i\in A_{f}(\Omega)}a_{i}x_{i}\right)\right\vert \\ &\leq\frac{1}{w(f)}\left\vert f_{1}\left(\sum_{i\in A_{f}(\Omega)}a_{i}x_{i}\right)\right\vert + \frac{1}{w(f)}\left\vert f_{2}\left(\sum_{i\in A_{f}(\Omega)}a_{i}x_{i}\right)\right\vert  \\  &\qquad + \frac{1}{w(f)}\left\vert \sum_{q\in B}f_{q}\left(\sum_{i\in A_{f}(\Omega)}a_{i}x_{i}\right)\right\vert + \frac{1}{w(f)}\left\vert\sum_{q\in C}f_{q}\left(\sum_{i\in A_{f}(\Omega)}a_{i}x_{i}\right)\right\vert  \end{align*} 
where, after applying the inductive hypothesis to the $g_{1}$ and $g_{2}$ terms and using the same argument as is used in the base case for the $g_{3}$ term, this quantity is bounded above by
\begin{align*} &\frac{\gamma}{w(f)}\left(1+\frac{3}{\sqrt{w(f_{1})}}\right)+\frac{\gamma}{w(f)}\left(1+\frac{3}{\sqrt{w(f_{2})}}\right) +\frac{3\gamma}{\sqrt{w(f)}}+\frac{1}{w(f)}\left\vert \sum_{q\in C}f_{q}\left(\sum_{i\in A_{f}(\Omega)}a_{i}x_{i}\right)\right\vert \\  &\qquad\leq \frac{\frac{13}{5}\gamma}{w(f)}+\frac{3\gamma}{\sqrt{w(f)}}+\frac{1}{w(f)}\left\vert\sum_{q\in C}f_{q}\left(\sum_{i\in A_{f}(\Omega)}a_{i}x_{i}\right)\right\vert \\ &\qquad\leq \frac{(3+\frac{13}{50})\gamma}{\sqrt{w(f)}}+ \frac{1}{w(f)}\left\vert\sum_{q\in C}f_{q}\left(\sum_{i\in A_{f}(\Omega)}a_{i}x_{i}\right)\right\vert \end{align*}
because $m_{1}\geq 100$. Now, note that there are sets $E_{0}\cup E_{1}\cup\ldots\cup E_{l}$ from Lemma \ref{aux_lemma} such that $\{q \mid f_{q}\in \mathcal{C}\}=E_{0}\cup E_{1}\cup\ldots\cup E_{l}$ and such that, for each $1\leq z\leq l$, we have that $(f_{q})_{q\in E_{z}}$ is $S_{n_{j_{1}}+\cdots+n_{j_{z}}}\ast\mathcal{A}_{3}$-admissible and $\max_{q\in E_{z}}d(m(f_{q}),\lambda_{z})\leq\delta$ for some $\lambda_{z}\in\overline{\mathcal{D}}$. Next, let $\Gamma_{z}=\{ i\in A_{f}(\Omega) \mid m_{t_{i}}=m(f_{q}) \text{ for some } q\in E_{z}\}$ for each $1\leq z\leq l$ and note that, if $i_{1},i_{2}\in \Gamma_{z}$, then we have that
\begin{align*} 2^{-N} &\leq d(m_{t_{i_{1}}},m_{t_{i_{2}}}) \\ &= d(m(f_{q_{1}}),m(f_{q_{2}}))\leq d(m(f_{q_{1}}),\lambda_{z})+d(\lambda_{z},m(f_{q_{2}})) \leq 2\delta \end{align*}
for the corresponding $q_{1},q_{2}\in E_{z}$, and this cannot happen. It follows that $\Gamma_{z}$ is at most a singleton, say $i_{z}$, for each $1\leq z\leq l$ so that
\begin{align*}\left\vert \sum_{q\in\mathcal{C}}f_{q}\left(\sum_{i\in A_{f}(\Omega)}a_{i}x_{i}\right)\right\vert &\leq \sum_{z=0}^{l} \sum_{q\in E_{z}} \left\vert f_{q}\left(\sum_{i \in A_{f}(\Omega)\setminus\{i_{z}\}}a_{i}x_{i}\right)\right\vert + \sum_{z=0}^{l}\sum_{q\in E_{z}}|f_{q}(a_{i_{z}}x_{i_{z}})| \\ &\leq \sum_{z=0}^{l}\sum_{q\in E_{z}}\frac{4\gamma}{\sqrt{w(f_{q})}} + \gamma\sum_{z=0}^{l}\sum_{q\in E_{z}}\frac{m_{t_{i_{z}}}}{w(f_{q})} \\ &\leq (l+1)\gamma + \sum_{z=0}^{l}\sum_{q\in E_{z}}\frac{4\gamma}{\sqrt{w(f_{q})}} \leq (l+5)\gamma  \end{align*}
where again we have used the inductive hypothesis and, by putting all of these pieces together, there is the estimate
\begin{align*} \left\vert f\left(\sum_{i\in\Omega}a_{i}x_{i}\right)\right\vert &\leq \gamma+\frac{(3+\frac{13}{50})\gamma}{\sqrt{w(f)}}+\frac{(l+5)\gamma}{w(f)} \\ &\leq \gamma+\frac{(3+\frac{13}{50})\gamma}{\sqrt{w(f)}}+\frac{(l+5)\gamma}{10^{l}\sqrt{w(f)}} \leq \gamma + \frac{(3+\frac{13}{50}+\frac{3}{5})\gamma}{\sqrt{w(f)}} \leq \gamma+ \frac{4\gamma}{\sqrt{w(f)}}\end{align*}
where, in particular, $\left\vert f\left(\sum_{i\in A_{f}(\Omega)}a_{i}x_{i}\right)\right\vert\leq \frac{4\gamma}{\sqrt{w(f)}} $ so that the proof is complete.
\end{proof}

\section{Rapidly increasing sequences and basic inequality}{\label{sec5}}

Rapidly increasing sequences (RISs) and the basic inequality are standard concepts used in most Banach spaces constructions using saturated norms, such as $X_\mathcal{D}$. The main utility of such sequences is that they can be found in every infinite-dimensional subspace of $X_\mathcal{D}$ and, via the basic inequality, the norm of their linear combinations is bounded above by the same linear combinations of basis vectors in the auxiliary space $X_\mathcal{D}^{\mathrm{aux},\delta}$. RISs first appeared in \cite{schlumprecht:1991} and the basic inequality has its roots in \cite{argyros:deliyanni:1997}. These tools will be essential in Section \ref{sec6} when we construct, in an arbitrary infinite-dimensional block subspace of $X_\mathcal{D}$, a tree of normalized vectors that witness the failure of the Lebesgue property.

We begin with a short technical lemma and we recall that $\supp(0)=\emptyset$ and that $\emptyset < A$ and $A<\emptyset$ for every subset $A\subset\N$.

\begin{lemma}{\label{basic_lemma}}
Fix $\delta>0$ and assume that $f_{1}<\ldots<f_{d}\in W$ with $2^{-\max\supp(f_{1})}\leq\frac{\delta}{2}$ are $(n_{j_{1}},\ldots,n_{j_{l}},\lambda)$-packed for some $\lambda\in\overline{\mathcal{D}}$. Next, suppose that there exist sets of functionals $\mathcal{F}_{1},\ldots,\mathcal{F}_{d}$ such that, for each $1\leq q\leq d$, $\mathcal{F}_{q}=\{g^{q}_{1},g^{q}_{2},g^{q}_{3}\}$ with $g^{q}_{r}\in \{0\}\cup W^{\text{aux},\delta}$, and:
\begin{enumerate}
\item $\cup_{r=1}^{3}\supp(g^{q}_{r})\subset\supp(f_{q})$ and $\supp(g^{q}_{1})<\supp(g^{q}_{2})<\supp(g^{q}_{3})$,
\item there is at most one $g^{q}_{\tilde{r}}\in\mathcal{F}_{q}$ such that $g^{q}_{\tilde{r}}$ is not a unit vector,
\item there is at most one $g^{q}_{r}\in\mathcal{F}_{q}$ whose support precedes the support of $g^{q}_{\tilde{r}}$, and
\item for $g^q_{\tilde r}\in \mathcal{F}_q$ that is not a unit vector,  $\vec w(g^{q}_{\tilde{r}})=\vec w(f_{q})$.
\end{enumerate}
Then, the functionals $(g^{q}_{r})_{r=1,q=1}^{3,d}$ are $(n_{j_{1}},\ldots,n_{j_{l}},\lambda)_{\delta}$-packed.
\end{lemma}
\begin{proof}
We proceed by induction on the length $l$ of the tuple $(m_{j_{1}},\ldots,m_{j_{l}})$. Let $l=1$ and assume that $f_{1}<\ldots<f_{d}\in W$ with $2^{-\max\supp(f_{1})}\leq\delta$ are $(n_{j_{1}},\lambda)$-packed for some $\lambda\in\overline{\mathcal{D}}$. Then, $(g^{q}_{r})_{r=1,q=1}^{3,d}\subset W^{\text{aux},\delta}_{k}$ for some $k\in\N$ are $S_{n_{j_{1}}}\ast\mathcal{A}_{3}$-admissible and we have that
\bes w(g^{q}_{\tilde{r}})=w(f_{q})\geq 2^{\max\supp(f_{q-1})}\geq 2^{\max\supp(\varphi)} \ees
for each $g^{q}_{\tilde{r}}\in\mathcal{F}_{q}$ with at least two preceding non-zero functionals, where $\varphi$ is the non-zero functional that is two before $g^{q}_{\tilde{r}}$. It follows that $(g^{q}_{r})_{r=1,q=1}^{3,d}$ are skipped very fast growing in the sense of Definition \ref{delta_packed}. Moreover, we have that
\bes d(m(g^{q}_{\tilde{r}}),\lambda)=d(m(f_{q}),\lambda) \leq 2^{-\max\supp(f_{q-1})} \leq 2^{-\max\supp(f_{1})}\leq\frac{\delta}{2}< \delta \ees
for $2\leq q\leq d$ so that $(g^{q}_{r})_{q=1,r=1}^{d,3}$ are $(n_{j_{1}},\lambda)_{\delta}$-packed.

Assume now that the result holds for tuples of length $l$ and let $f_{1}<\ldots<f_{d}\in W$ be $(n_{j_{1}},\ldots,n_{j_{l}},n_{j_{l+1}},\lambda)$-packed for some $\lambda\in\overline{\mathcal{D}}$. Let $\mathcal{F}_{1},\ldots,\mathcal{F}_{d}$ be sets of at most three auxiliary weighted functionals as in the statement of the lemma and let $F_{1}<\ldots<F_{M}$ be a partition of $\{1,\ldots,d\}$ as in Definition \ref{big_packed}. Note that $(f_{q})_{q\in F_{1}}$ are $(n_{j_{1}},\ldots,n_{j_{l}},\lambda_{1})$-packed for some $\lambda_{1}\in\overline{\mathcal{D}}$. It therefore follows from the inductive hypothesis that $(g^{q}_{r})_{r=1,q\in F_{1}}^{3}$ are $(n_{j_{1}},\ldots,n_{j_{;}},\lambda_{1})_{\delta}$-packed. Moreover, for each $g^{q}_{\tilde{r}}$ with $q\in F_{p}$ for $p>1$, we have that
\begin{align*} d(m(g^{q}_{\tilde{r}}),\lambda)&=d(m(f_{q}),\lambda) \leq d(m(f_{q}),\lambda_{p})+d(\lambda_{p},\lambda) \\ &\leq 2\cdot 2^{-\max\supp(f_{\max(F_{p-1})})}\leq 2\cdot 2^{-\max\supp(f_{1})}\leq \delta \end{align*}
and, because $(g^{q}_{r})_{q=1,r=1}^{d,3}$ are skipped very fast growing in the sense of Definition \ref{big_delta_packed} by the same reasoning that is used in the base case, it follows that these functionals are $(n_{j_{1}},\ldots,n_{j_{l}},n_{j_{l+1}},\lambda)_{\delta}$-packed so that the proof is complete.
\end{proof}

We now recall from \cite{argyros:motakis:2020} the notion of a rapidly increasing sequence (RIS).

\begin{definition}{\label{RIS}}
Let $C\geq 1$, let $I\subset\N$ be an interval, and let $(j_{i})_{i\in I}$ be a strictly increasing sequence of positive integers. Then, a block sequence $(x_{i})_{i\in I}$ is said to be a $(C,(j_{i})_{i\in I})$-RIS if the following three statements hold:
\begin{enumerate}
\item $\|x_{i}\|\leq C$ for each $i\in\N$.
\item $\max\supp(x_{i-1})<\sqrt{m_{j_{i}}}$ for each $i\in I\setminus\{\min(I)\}$.
\item $|f(x_{i})|\leq\frac{C}{w(f)}$ for each $i\in I$ and for every $f\in W$ with $w(f)<m_{j_{i}}$.
\end{enumerate}
\end{definition}

We recall the steps of a standard argument that shows that there exists a RIS in every block subspace of $X_\mathcal{D}$. For some of the details, we refer to specific parts of \cite{argyros:motakis:2020}. The proposition below is proved word-for-word as \cite[Proposition 6.3]{argyros:motakis:2020}, and it requires only two things: Theorem \ref{l1_spreading_model} (in place of \cite[Proposition 4.1]{argyros:motakis:2020}) and that, for every $C>1$, $m_j^{-1}C^{n_j}\to\infty$. The latter is true because, by condition \ref{condition n_j} on page \pageref{condition n_j},
\begin{align*} \log\left(\frac{C^{n_{j}}}{m_{j}}\right) >\frac{4\log(C)}{\log(100)}n_{j-1}\log(m_{j})-\log(m_{j})=\log(m_{j})\left(\frac{4\log(C)n_{j-1}}{\log(100)}-1\right). \end{align*}

\begin{proposition}
\label{scc existene in block subspace}
Let $Y$ be a block subspace of $X_\mathcal{D}$. Then, for every $n\in\mathbb{N}$, $\varepsilon>0$, and $\delta>0$ there exists a $(n,\varepsilon)$ special convex combination $x= \sum_{i=1}^mc_ix_i$ such that $\|x\| > 1/(1+\delta)$ and $x_1,\ldots,x_m$ are block vectors in $Y$ of norm at most one.
\end{proposition}

The next proposition is also proved verbatim as \cite[Proposition 6.4]{argyros:motakis:2020} and it uses Remark \ref{weight-to-admissibility condition}.

\begin{proposition}
\label{scc RIS property}
Let $j\in\N$, $\varepsilon>0$, and $x = \sum_{i=1}^mc_ix_i$ be an $(n_j,\varepsilon)$-special convex combination such that $\|x_i\|\leq 1$, $1\leq i\leq m$. Then, for every $f\in W$ such that $w(f)<m_j$,
\[|f(x)| \leq \frac{1+2w(f)\varepsilon}{w(f)}.\]
\end{proposition}

The desired conclusion follows immediately from Proposition \ref{scc existene in block subspace} and Proposition \ref{scc RIS property}.

\begin{corollary}
\label{RIS exist}
For every $C>1$ and in every block subspace $Y$ of $X_\mathcal{D}$ there exists a $(C,(j_i)_{i\in\N})$-RIS $(x_i)$, with $\|x_i\|\geq 1$, for $i\in\N$.
\end{corollary}

 We now prove the basic inequality that transfers the action of a weighted functional $f\in W$ to the action of auxiliary weighted functionals on appropriate unit vectors.

\begin{proposition}{\label{basic_inequality}}
Let $(x_{i})_{i\in I}$ be a $(C,(j_{i})_{i\in I})$-RIS with $2^{-\min\supp(x_{\min(I)})}\leq\delta$. Then, for every functional $f=\frac{1}{m_{j_{1}}\cdots m_{j_{l}}}\sum_{q=1}^{d}f_{q}\in W$, there exist functionals $h\in \{0\}\cup W^{\text{aux},\delta}_{0}$ and $g\in W^{\text{aux},\delta}$ with $\vec w(g)=\vec w(f)$ such that
\bes \left\vert f\left(\sum_{i\in I}a_{i}x_{i}\right)\right\vert \leq C\left(1+\frac{1}{\sqrt{m_{j_{\min(I)}}}}\right)\left\vert (h+g)\left(\sum_{i\in I}a_{i}e_{t_{i}}\right)\right\vert \ees
for every choice $(a_{i})_{i\in I}$ of scalars where $t_{i}=\max\supp(x_{i})$ for each $i\in I$.
\end{proposition}
\begin{proof}
We  proceed by induction on the level $k$ norming sets $W_{k}$ and the base case is the level $k=0$ set $W_{0}=\{\pm e_{i}^{*} \mid i\in\N\}$. Indeed, if $f=e^{*}_{i_{0}}\in W_{0}$, then
\bes \left\vert e^{*}_{i_{0}}\left(\sum_{i\in I}a_{i}x_{i}\right)\right\vert \leq |a_{i_{1}}e^{*}_{i_{0}}(x_{i_{1}})| \leq C|a_{i_{1}}|= C|e^{*}_{t_{i_{1}}}(a_{t_{i_{1}}}e_{t_{i_{1}}})| \ees
where $i_{0}\in\supp(x_{i_{1}})$ so that the desired estimate holds with $h=\{0\}$ and with $g=e^{*}_{t_{i_{1}}}$. Note that, as part of the inductive hypothesis, we have that $\supp(h),\supp(g)\subset\{t_{i}\}_{i\in I}$ and that: either $h=0$ or $h\in W^{\text{aux},\delta}_{0}$, $\supp(h)<\min\supp(g)$.

Assume now that the desired estimate holds for functionals in the level $k$ norming set $W_{k}$ and consider the functional
\bes f=\frac{1}{m_{j_{1}}\cdots m_{j_{l}}}\sum_{q=1}^{d}f_{q}\in W_{k+1}\setminus W_{k} \ees
which is to say that $f_{1}<\ldots<f_{d}$ are in $W_{k}$ and are $(n_{j_{1}},\ldots,n_{j_{l}},\lambda)$-packed for some $\lambda\in\overline{\mathcal{D}}$. We omit if necessary the functionals $f_{q}$ that are below the support of $x_{\min(I)}$ and it then follows that 
\bes \frac{1}{\delta}\leq 2^{\min\supp(x_{\min(I)})}\leq 2^{\max\supp(f_{1})} \implies 2^{-\max\supp(f_{1})}\leq\delta \ees
and we will use this observation later. First, define the quantity $i_{0}$ by
\bes i_{0}=\max\{i\in I \mid m_{j_{1}}\cdots m_{j_{l}}\geq m_{j_{i}}\} \ees
if it exists. If there does exist such an $i_{0}$, then choose $i_{1}\in[\min(I),i_{0}]$ so that $|a_{i_{1}}|=\max\{|a_{i}|\}_{i=\min(I)}^{i_{1}}$ and define the functional $h=\text{sign}(a_{i_{1}})e^{*}_{t_{i_{1}}}\in W^{\text{aux},\delta}_{0}$. Note that, if $\min(I)=i_{0}=i_{1}$, then
\bes \left\vert f\left(\sum_{i\leq i_{0}}a_{i}x_{i}\right)\right\vert=|f(a_{i_{1}}x_{i_{1}})|\leq C|a_{i_{1}}|=Ch(a_{i_{1}}e_{t_{i_{1}}}) \ees
and, if $\min(I)<i_{0}$, then
\begin{align*} \left\vert f\left(\sum_{i\leq i_{0}}a_{i}x_{i}\right)\right\vert &\leq |a_{i_{0}}f(x_{i_{0}})|+\sum_{i<i_{0}}|a_{i}f(x_{i})| \leq C|a_{i_{0}}|+\sum_{i<i_{0}}\left\vert\frac{a_{i}}{w(f)}\sum_{q=1}^{d}f_{q}(x_{i})\right\vert \\ &\leq C|a_{i_{1}}|+\frac{|a_{i_{1}}|}{m_{j_{i_{0}}}}\sum_{i<i_{0}}\sum_{q=1}^{d}\left\vert f_{q}\left(\sum_{r\in\supp(x_{i})}b_{r}e_{r}\right)\right\vert \\ &\leq C|a_{i_{1}}|+\frac{|a_{i_{1}}|}{m_{j_{i_{0}}}}\sum_{i<i_{0}}\sum_{r\in\supp(x_{i})}|b_{r}| \leq C|a_{i_{1}}|+\frac{C|a_{i_{1}}|}{m_{j_{i_{0}}}}\max\supp(x_{i_{0}-1}) \\ &\leq C\left(1+\frac{1}{\sqrt{m_{j_{i_{0}-1}}}}\right)|a_{i_{1}}| \leq C\left(1+\frac{1}{\sqrt{m_{j_{\min(I)}}}}\right)h\left(\sum_{i\leq i_{0}}a_{i}e_{t_{i}}\right)\end{align*}
and if, on the other hand, there is no such number $i_{0}$, then the sum $\sum_{i\leq i_{0}}a_{i}x_{i}$ is empty so we take $h=0$.

Next, let $\tilde{I}=\{i\in I \mid i>i_{0}\}$ if such an $i_{0}$ exists and put $\tilde{I}=I$ otherwise. We then define the subsets
\begin{align*}
A&=\{ i\in\tilde{I} \mid \supp(x_{i})\cap\supp(f_{q})\neq\emptyset \text{ for at most one }q\} \\ I_{q}&=\{ i\in A \mid \supp(x_{i})\cap\supp(f_{q})\neq\emptyset\} \\ D&=\{q\mid I_{q}\neq\emptyset\} \\ B&=\tilde{I}\setminus A \end{align*}
so that the intervals $I_{q}$ are pairwise disjoint and $\supp(x_{i})$ intersects the support of at least two $f_{q}$ if $i\in B$. It follows by the inductive hypothesis that, for each $q\in D$, there exist functionals $h_{q}\in \{0\}\cup W^{\text{aux},\delta}_{0}$ and $g_{q}\in W^{\text{aux},\delta}$ with $\vec w(g_{q})=\vec w(f_{q})$ so that 
\bes \left\vert f_{q}\left(\sum_{i\in I_{q}}a_{i}x_{i}\right)\right\vert\leq  C\left(1+\frac{1}{\sqrt{m_{j_{\min(I)}}}}\right)\left\vert (h_{q}+g_{q})\left(\sum_{i\in I_{q}}a_{i}e_{t_{i}}\right)\right\vert \ees
where $\supp(h_{q}),\supp(g_{q})\subset\supp(f_{q})$ and we have that: either $h_{q}=0$ or $h_{q}\in\{e_{t_{i}}\}_{i\in I_{q}}$ and $t_{i}<\min\supp(g_{q})$. Now, replace $h_{q}+g_{q}$ by $\text{sign}\left[(h_{q}+g_{q})\left(\sum_{i\in I_{q}}a_{i}e_{t_{i}}\right)\right](h_{q}+g_{q})$ if necessary and note that
\begin{align*} \left\vert f\left(\sum_{i\in\tilde{I}}a_{i}x_{i}\right)\right\vert &\leq \sum_{i\in B}|f(x_{i})|+\frac{1}{w(f)}\sum_{q\in D}\left\vert f_{q}\left(\sum_{i\in I_{q}}a_{i}x_{i}\right)\right\vert  \\ &\leq\frac{C}{w(f)}\sum_{i\in B}|a_{i}|+C\left(1+\frac{1}{\sqrt{m_{j_{\min(I)}}}}\right)\frac{1}{w(f)}\sum_{q\in D}(h_{q}+g_{q})\left(\sum_{i\in I_{q}}a_{i}e_{t_{i}}\right) \\ &\leq C\left(1+\frac{1}{\sqrt{m_{j_{\min(I)}}}}\right)\frac{1}{w(f)}\left[\sum_{i\in B}\tilde{h}_{i}+\sum_{q\in D}(h_{q}+g_{q}) \right]\left(\sum_{i\in\tilde{I}}a_{i}e_{t_{i}}\right)\end{align*}
where we have used the third condition from Definition \ref{RIS} and where $\tilde{h}_{i}=\text{sign}(a_{i})e_{t_{i}}^{*}$ for each $i\in B$. It therefore follows that
\bes \left\vert f\left(\sum_{i\in I}a_{i}x_{i}\right)\right\vert \leq C\left(1+\frac{1}{\sqrt{m_{j_{\min(I)}}}}\right)\left[(h+g)\left(\sum_{i\in I}a_{i}e_{t_{i}}\right)\right] \ees
where: either $h=0$ or $h\in\{e_{t_{i}}\}_{i=\min(I)}^{i_{0}}$ and $\supp(h)=t_{i_{1}}<\min\supp(g)$. Indeed, we have that
\bes g=\frac{1}{w(f)}\left(\sum_{i\in B}\tilde{h}_{i}+\sum_{q\in D}(h_{q}+g_{q})\right) \ees
so that $\vec w(f)=\vec w(g)$. It remains to prove that $g\in W^{\text{aux},\delta}$. To that end, define for each $i\in B$ the quantity $q_{i}=\max\{1\leq q\leq d \mid \min\supp(f_{q})\leq \max\supp(x_{i})=t_{i}\}$ and note that the correspondence $i\mapsto q_{i}$ is strictly increasing. If $q=q_{i}$ for some $i\in B$, then write $\mathcal{F}_{q}=\{h_{q},g_{q},\tilde{h}_{i}\}$. If $q\neq q_{i}$ for any $i\in B$, then write $\mathcal{F}_{q}=\{h_{q},g_{q}\}$. Note that $\mathcal{F}_{q}$ may consist only of $\tilde{h}_{i}$ or may be empty if $q\notin D$. It therefore follows from Lemma \ref{basic_lemma} (applied to $(f_{q})_{q\in D}$) that $g\in W^{\text{aux},\delta}$ and this completes the proof.
\end{proof}

\section{The failure of the Lebesgue property in subspaces of $X_\mathcal{D}$}{\label{sec6}}
In this section, we put together the results of the previous two sections in order to prove that every subspace of $X_{\mathcal{D}}$ contains a normalized collection $(y_\lambda)_{\lambda\in\mathcal{D}}$ that satisfies the estimates \eqref{intro ellinfty estimate}, and thus, witnesses the failure of the Lebesgue property.

\begin{theorem}{\label{main_result}}
Let $Y$ be an infinite-dimensional block subspace of $X_\mathcal{D}$. Then, there exists a normalized collection $(y_\lambda)_{\lambda\in\mathcal{D}}$ in $Y$ that is block with the lexicographical order of $\mathcal{D}$ such that the following holds: for every $N\in\N$ and $(\mu_\lambda)_{\lambda\in\{0,1\}^N}$ in $\mathcal{D}$ such that, for $\lambda\in\{0,1\}^N$, $\mu_\lambda\geq \lambda$ and for any scalars $(a_\lambda)_{\lambda\in\{0,1\}^N}$,
\[\max_{\lambda\in\{0,1\}^N}|a_\lambda|\leq \Big\|\sum_{\lambda\in\{0,1\}^N}a_\lambda y_{\mu_{\lambda}}\Big\| \leq 3\max_{\lambda\in\{0,1\}^N}|a_\lambda|.\]
In particular, the Lebesgue property fails in every infinite-dimensional closed subspace of $X_{\mathcal{D}}$.
\end{theorem}

\begin{proof}
Let $Y\subset X_{\mathcal{D}}$ and let $Z\subset Y$ be a block subspace. Then, fix $C> 1$ and let $(z_{i})$ be an infinite $(C,(j_{i})_{i\in\N})$-RIS in $Z$ with $1\leq\|z_{i}\|$ for each $i\in\N${, which exists by Corollary \ref{RIS exist}. By passing to a subsequence, we may assume that the conclusion of Theorem \ref{l1_spreading_model} is satisfied.} Recall that the nodes of $\mathcal{D}$ are in lexicographical order with respect to $\mathcal{M}$ and choose for each weight $m_{j}$ a number $\ep_{j}\in\left(0,\frac{1}{3m_{j}2^{N_{j}}}\right)$, where {$N_j = |\phi(m_j)|$.}

Now, let $F_{1}<F_{2}<\ldots$ be subsets of $\N$ so that $\{\min\supp(z_{i}) \mid i\in F_{j}\}$ is a maximal set in $S_{n_{j}}$ with respect to inclusion for each $j\in\N$. It follows from Proposition \ref{basic-scc-maximal} that there exist coefficients $(c^{j}_{i})_{i\in F_{j}}$ so that the vector $\sum_{i\in F_{j}}c^{j}_{i}e_{i}$ is an $(n_{j},\frac{\ep_{j}}{2})$-basic scc. In particular, we have that $x_{m_j}=\sum_{i\in F_{j}}c^{j}_{i}z_{i}$ is an $(n_{j},\frac{\ep_{j}}{2})$-scc and we note also that
\bes \|x_{m_j}\|=\left\Vert\sum_{i\in F_{j}}c^{j}_{i}z_{i}\right\Vert \geq \frac{m_{1}-1}{m_{1}m_{j}}\sum_{i\in F_{j}}c^{j}_{i}=\frac{m_{1}-1}{m_{1}m_{j}}\ees
from Theorem \ref{l1_spreading_model}. Let $\tilde{x}_{m_j}=m_{j}x_{j}$ and define, for $\lambda\in\mathcal{D}$, $y_{\lambda}= \tilde x_{\phi^{-1}(m_j)}/\|\tilde x_{\phi^{-1}(m_j)}\|$ so that $(y_{\lambda})$ is a normalized sequence in $Z$ that is block in the lexicographical order of $\mathcal{D}$.

Now, fix $N\in\N$ and, for $\lambda\in\{0,1\}^N$, let $\mu_\lambda\in\mathcal{D}$ such that $\mu_\lambda\geq \lambda$. Because $(y_\lambda)_{\lambda\in\mathcal{D}}$ is lexicographically block, there are $2^N\leq j_1<\cdots<j_{2^N}$ such that $\phi(\{m_{j_1},\ldots,m_{j_{2^N}}\}) = \{\mu_\lambda:\lambda\in\{0,1\}^N\}$. In particular, $\min\supp(y_{\mu_\lambda})\geq N$, $\lambda\in\{0,1\}^N$. Furthermore, for $\lambda\in\{0,1\}^N$, $\mu_\lambda\geq\lambda$, and thus, for $1\leq k\neq k'\leq 2^N$, $d(m_{j_k},m_{j_{k'}}) \geq 2^{-N}$. Then, it follows that
\bes \left\Vert \sum_{\lambda\in\{0,1\}^N}y_{\mu_\lambda}\right\Vert =\mathscr{F}\left(\sum_{\lambda\in\{0,1\}^N}y_{\mu_\lambda}\right) =\mathscr{F}\left(\sum_{k=0}^{2^{N}-1}\frac{1}{\|\tilde{x}_{m_{j_{k}}}\|}m_{j_{k}}\sum_{i\in F_{j_{k}}}c^{j_{k}}_{i}z_{i}\right) \qquad (\#) \ees
for some norming functional $\mathscr{F}\in W$ and, by Proposition \ref{basic_inequality}, this quantity is bounded above by
\bes C\left\vert (h+g)\left(\sum_{k=0}^{2^{N}-1}\frac{1}{\|\tilde{x}_{m_{j_{k}}}\|}m_{j_{k}}\sum_{i\in F_{j_{k}}}c^{j_{k}}_{i}e_{\rho(i)}\right)\right\vert \qquad (\#\#) \ees
where $h\in W^{\text{aux},\delta}_{0}\cup\{0\}$, $g\in W^{\text{aux},\delta}$ with $\vec w(g)=\vec w(\mathscr{F})$, and where $\rho(i)=\max\supp(z_{i})$ for each $i\in\N$. Note that $\sum_{i\in F_{j_{k}}}c^{j_{k}}_{i}e_{\rho(i)}$ is for each $0\leq k\leq 2^{N}-1$ an $(n_{j_{k}},\ep_{j_{k}})$-basic scc by the observation that immediately follows Proposition \ref{basic-scc-maximal}, and note also that
\[\frac{m_{j_{k}}}{\|\tilde{x}_{m_{j_{k}}}\|}=\frac{1}{\|x_{m_{j_{k}}}\|}\leq \frac{m_{1}m_{j_{k}}}{m_{1}-1}.\]
It therefore follows that $(\#\#)$ is bounded above by
\bes C\max_{0\leq k\leq 2^{N}-1}\frac{m_{j_{k}}\ep_{j_{k}}}{\|\tilde{x}_{j_{k}}\|} +C\left\vert g\left(\sum_{k=0}^{2^{N}-1}\frac{1}{\|\tilde{x}_{j_{k}}\|}m_{j_{k}}\sum_{i\in F_{j_{k}}}c^{j_{k}}_{i}e_{\rho(i)}\right)\right\vert \qquad (\#\#\#)  \ees
and lastly, by Proposition \ref{c0_type_estimate}, we note that $(\#\#\#)$ is bounded above by
\bes C\frac{m_{1}m_{j_{k}}}{m_{1}-1}\frac{1}{3m_{j_{k}} 2^{N_{j_{k}}}}+C\left(1+\frac{4}{\sqrt{w(g)}}\right)\max_{0\leq k\leq 2^{N}-1}\frac{1}{\|\tilde{x}_{j_{k}}\|} \leq C\frac{100}{99}\left(1+\left(1+\frac{2}{5}\right)\right)\leq 3, \ees
for $C>1$ sufficiently close to one. The estimate for arbitrary scalars coefficients follows from the 1-unconditionality of the basis of $X_\mathcal{D}$.
\end{proof}

\bibliographystyle{plain}
\bibliography{bibliography}

\end{document}